\documentclass[a4paper,12pt,leqno]{article}

\usepackage{amssymb,amsmath}
\usepackage[abbrev]{amsrefs}
\usepackage{xy}
\xyoption{all}
\usepackage{graphicx}

\newtheorem{theorem}{Theorem}[section]
\newtheorem{lemma}[theorem]{Lemma}
\newtheorem{corollary}[theorem]{Corollary}
\numberwithin{equation}{theorem}
\newtheorem{exm}[theorem]{Example}
\newtheorem{xproof}{{\it Proof. }}

\newtheorem{xrem}{Remark.}

\newenvironment{example}{\begin{exm}\em}{\end{exm}}
\newenvironment{proof}{\begin{xproof}\em}{\end{xproof}}
\newenvironment{remark*}{\begin{xrem}\em}{\end{xrem}}
\def\qedhere{\hspace{0.3cm}{\rule{1ex}{2ex}}}

\newcommand\Max{\operatorname{Max}}
\newcommand\eg{e.g.}
\newcommand\st{\mid}
\newcommand\cf{\textrm{cf.}}
\newcommand\topology{\operatorname{\Omega}}
\newcommand\upsegment{{\uparrow}}
\newcommand\Vect{\textit{Vect}}
\newcommand\Set{\textit{Set}}
\newcommand\Top{\textit{Top}}
\newcommand\opp[1]{{#1}^{\textrm{op}}}
\newcommand\ident{\mathrm{id}}
\newcommand\CC{\mathbb{C}}
\newcommand\RR{\mathbb{R}}
\newcommand\NN{\mathbb{N}}
\newcommand\supp{\operatorname{supp}}
\newcommand\Sub{\operatorname{Sub}}
\newcommand\norm[1]{\| #1\|}
\newcommand\sections{C}
\newcommand\linspan[1]{\left\langle #1\right\rangle}
\newcommand\Bun{\textit{Bun}}
\newcommand\LBun{\textit{LinBun}}
\newcommand\eval{\operatorname{eval}}
\newcommand\ps[1]{\widetilde{#1}}
\newcommand\Gr{\operatorname{\boldsymbol{Gr}}}
\newcommand\GL{\operatorname{\boldsymbol{GL}}}
\newcommand\image{\operatorname{Im}}
\newcommand\ft[1]{\check{#1}}
\newcommand\cb[2]{{\boldsymbol U\!}_{#2}(#1)}
\newcommand\pr{\hat}
\newcommand\fdom[1]{{\mathfrak D}_{#1}}
\newcommand\pf[1]{{\mathfrak p}_{#1}}
\newcommand\ff[1]{{\mathfrak f}_{#1}}
\newcommand\interior{\operatorname{int}}

\begin{document}

\title{Open quotients of trivial vector bundles\thanks{Work funded by FCT/Portugal through projects EXCL/MAT-GEO/0222/2012 and PEst-OE/EEI/LA0009/2013, and by COST (European Cooperation in Science and Technology) through COST Action MP1405 QSPACE.}}
\author{{\sc Pedro Resende} and {\sc Jo{\~a}o Paulo Santos}}

\date{~}

\maketitle

\vspace*{-1cm}
\begin{abstract}
Given an arbitrary topological complex vector space $A$, a \emph{quotient vector bundle} for $A$ is a quotient of a trivial vector bundle $\pi_2:A\times X\to X$ by a fiberwise linear continuous open surjection. We show that this notion subsumes that of a Banach bundle over a locally compact Hausdorff space $X$. Hyperspaces consisting of linear subspaces of $A$, topologized with natural topologies that include the lower Vietoris topology and the Fell topology, provide classifying spaces for various classes of quotient vector bundles, in a way that generalizes the classification of locally trivial vector bundles by Grassmannians. If $A$ is normed, a finer hyperspace topology is introduced that classifies bundles with continuous norm, including Banach bundles, and such that bundles of constant finite rank must be locally trivial.
\\
\vspace*{-2mm}~
\vspace*{-2mm}~
\end{abstract}

\tableofcontents

\section{Introduction}

The reduced C*-algebra $C^*_{\mathrm{red}}(G)$ of a locally compact Hausdorff \'etale groupoid $G$ (see \cites{RenaultLNMath,Paterson}), whose elements can be regarded as complex valued functions on $G$, can be ``twisted'' by 
considering instead the reduced C*-algebra $C^*_{\mathrm{red}}(\pi)$ of a Fell line bundle $\pi:E\to G$, where now the elements of the algebra are identified with sections of the bundle; see  \cites{Renault,Kumjian98}. This 
provides one way of generalizing to C*-algebras the notion of Cartan subalgebra of a Von Neumann algebra \cite{FeldmanMooreI-II}: \emph{Cartan pairs} $(A,B)$, consisting of a C*-algebra $A$ and a suitable abelian subalgebra $B$, correspond bijectively to a certain class of \'etale groupoids with Fell line bundles on them \cite{Renault}.

The motivation for the present paper stemmed from studying the construction of a Fell bundle from a Cartan pair, in particular in an attempt to generalize the class of groupoids to which it applies by taking into account that both a C*-algebra $A$ and an \'etale groupoid $G$ have associated quantales $\Max A$ \cites{KR,MP1,MP2} and $\topology(G)$ \cite{Re07}, respectively, the former consisting of all the closed linear subspaces of $A$ and the latter being the topology of $G$. In doing so it became evident that it is useful to study bundles whose construction is based on a preexisting object of global sections, such as the C*-algebra $A$ of a Cartan pair, in a way that in fact is independent of the algebraic structure of groupoids, but which instead applies to Banach bundles to begin with, and indeed to more general bundles. So the original endeavour has naturally been split into several parts, of which the present paper is the first one, where no further mention of C*-algebras or groupoids will be made. Instead the bulk of this paper will deal with completely general topological vector spaces, and on occasion locally convex or normed spaces.

If $A$ is a topological vector space and $X$ is any topological space, then $\pi_2:A\times X\to X$ is a trivial vector bundle on $X$. The vector bundles studied in this paper, termed \emph{quotient vector bundles}, consist of those bundles $\pi:E\to X$ that arise as quotients of trivial bundles by a fiberwise linear continuous open surjection $q$:
\[
\xymatrix{
A\times X\ar@{->>}[rr]^q\ar[rd]_{\pi_2}&&E\ar[dl]^{\pi}\\
&X
}
\]
We shall see that any Banach bundle $\pi:E\to X$ (in particular, any finite rank locally trivial vector bundle) on a locally compact Hausdorff space is of this kind, where $A$ can be taken to be the space $\sections_0(\pi)$ of continuous sections vanishing at infinity.

Although now without any quantale structure (since $A$ is not even an algebra), the set $\Max A$ of all the closed linear subspaces of $A$ plays an important role: equipped with suitable topologies it provides a notion of classifying space for quotient vector bundles. Concretely, these are obtained by pullback along continuous maps
\[
\kappa: X\to\Max A
\]
(or, even more generally, maps into $\Sub A$, the space of all the linear subspaces) of a universal bundle $\pi_A:E_A\to\Max A$:
\[
\xymatrix{
E_A\times_{\Max A} X\ar[rr]^-{\pi_1}\ar[d]_{\pi_2}&&E_A\ar[d]^{\pi_A}\\
X\ar[rr]_{\kappa}&&\Max A
}
\]
In this paper we study three topologies on $\Max A$. Perhaps surprisingly, two of them are well known hyperspace topologies:
\begin{itemize}
\item
The (relative) lower Vietoris topology \cites{NT96,Vietoris} classifies all the quotient vector bundles with Hausdorff fibers (the restriction on the fibers disappears if we use $\Sub A$ instead of $\Max A$).
\item The topology of Fell \cite{Fell62} classifies the quotient vector bundles whose zero section is closed --- at least provided both $A$ and $X$ are first countable.
\end{itemize}
If $A$ is normed its quotient vector bundles $\pi:E\to X$ are naturally equipped with an upper semicontinuous norm $\norm~ : E\to\RR$. In this case a third topology on $\Max A$, referred to as the \emph{closed balls topology}, coarser than the (full) Vietoris topology but finer than the Fell topology, classifies the bundles for which the norm on $E$ is continuous. In particular, if $A$ is a Banach space it classifies Banach bundles over first countable Hausdorff spaces.

Quotient vector bundles are not necessarily locally trivial, and no classification of the locally trivial ones for arbitrary $A$ is provided in general. But something can be said about quotient vector bundles whose fibers are all of the same finite dimension $d$. These are classified by the subspace $\Max_d A\subset \Max A$ that contain the closed linear subspaces of codimension $d$ in $A$. We show that such bundles are locally trivial if $A$ is normed and $\Max_d A$ is equipped with the closed balls topology. As a corollary, this yields a fact that is stated but not proved in \cite{FD1}*{p.\ 129}, namely that Banach bundles on locally compact Hausdorff spaces whose fibers are of constant finite dimension are locally trivial.

The role played by our spaces $\Max A$ with respect to quotient vector bundles is analogous to that of Grassmannians for locally trivial vector bundles, and at the end of the paper we provide a detailed comparison between them, in particular showing explicitly that $\Gr(d,\CC^n)$, the Grassmannian of $d$-dimensional subspaces of $\CC^n$, is homeomorphic to $\Max_{n-d}\CC^n$.

\section{Preliminaries}\label{sec:defs}

We begin by fixing basic terminology and notation.

\paragraph{Topological vector spaces.}

By a \emph{topological vector space} will be meant a complex topological vector space without assuming any topological separation axioms except where specified otherwise.
We recall that if the topology of a topological vector space is $T_1$ then it is Hausdorff (in fact completely regular), and that a Hausdorff vector space is finite dimensional if and only if it is locally compact. Moreover, the topology of a finite dimensional Hausdorff vector space is necessarily the Euclidean topology.
We define the following notation, for an arbitrary topological vector space $A$:
\begin{itemize}
\item $\Sub A$ is the set of all the linear subspaces of $A$;
\item $\Max A$ is the subset of $\Sub A$ consisting of all the topologically closed linear subspaces.
\end{itemize}
The notation $\Max A$ is borrowed from quantale theory, where $\Max A$, equipped with a natural quantale structure, plays the role of spectrum of a C*-algebra $A$. See \cites{Mulvey-enc,MP1,MP2,KR}.

For each subset $S\subset A$, we denote by $\linspan S$ the linear span of $S$. If $S$ is a finite subset $\{a_1,\ldots,a_n\}$ we also write $\linspan{a_1,\ldots,a_n}$ instead of $\linspan S$.

\begin{example}
$\Sub A$ and $\Max A$ may coincide, for instance if $A$ has the discrete topology, or if $A$ is a Hausdorff finite dimensional vector space as in the following examples:
\begin{itemize}
\item $\Sub\CC=\Max\CC=\{\{0\},\CC\}$;
\item $\Sub \CC^2 = \Max \CC^2=\{\{0\},\CC^2\}\cup\CC P^1$;
\item $\Sub \CC^n=\Max \CC^n=\{\{0\},\CC^n\}\cup\coprod_{r=1}^{n-1}\Gr(r,\CC^n)$.
\end{itemize}
\end{example}

\paragraph{Bundles.}

Let $X$ be a topological space. By a \emph{bundle} on $X$ will always be meant a topological space $E$ equipped with a continuous surjection
\[\pi:E\to X\;,\]
which is referred to as the \emph{projection} of the bundle.
For each element $x\in X$ we refer to the subspace $\pi^{-1}(\{x\})\subset E$ as the \emph{fiber} of $\pi$ over $x$, and we use the following notation for all $x\in X$ and open sets $U\subset X$:
\begin{eqnarray*}
E_x&=&\pi^{-1}(\{x\})\;,\\
E_U&=&\pi^{-1}(U)\;,\\
\pi_U&=&\pi\vert_{E_U}:E_U\to U\;.
\end{eqnarray*}

If $\rho:F\to X$ and $\pi:E\to X$ are bundles, by a \emph{map} of bundles over $X$,
\[
h:\rho\to \pi\;,
\]
will be meant a continuous map $h:F\to E$ such that $\pi\circ f=\rho$. The category of bundles over $X$ and their maps is denoted by $\Bun(X)$.

If $\pi:E\to X$ is a bundle, $Y$ is a topological space and $g:Y\to X$ is a continuous map, the pullback $E\times_X Y$ together with its second projection $\pi_2:E\times_X Y\to Y$ defines a bundle $g^*(\pi)$ on $Y$:
\[
\xymatrix{
E\times_X Y\ar[rr]^{\pi_1}\ar[d]_{g^*(\pi)=}^{\pi_2}&&E\ar[d]^{\pi}\\
Y\ar[rr]_{g}&&X\;.
}
\]

\paragraph{Continuous sections.}

A \emph{continuous section} of a bundle $\pi:E\to X$ is a continuous function $s:X\to E$ such that $\pi\circ s$ is the identity on $X$. The set of all the continuous sections of the bundle is denoted by $\sections(\pi)$. We say that the bundle \emph{has enough sections} if for all $e\in E_x$ there is a continuous section $s$ such that $s(x)=e$.

Any map $h:\rho\to\pi$ in $\Bun(X)$ induces a mapping \[h_*:\sections(\rho)\to\sections(\pi)\] given by postcomposition,
\[
\xymatrix{
F\ar[rr]^h&&E\\
&X\ar[ur]_{h_*(s)=h\circ s}\ar[ul]^{s}
}
\]
and thus we obtain a functor $(~)_*:\Bun(X)\to\Set$.
In addition, given a pullback
\[
\xymatrix{
E\times_X Y\ar[rr]^{\pi_1}\ar[d]_{\pi_2}&&E\ar[d]^{\pi}\\
Y\ar[rr]_{g}&&X
}
\]
there is a mapping
$g^*:\sections(\pi)\to\sections(\pi_2)$, defined for all $s\in \sections(\pi)$ and $y\in Y$ by
\[
g^*(s)(y) = \bigl(s(g(y)),y\bigr)\;,
\]
so we obtain a contravariant functor $(~)^*:\opp{(\Top/X)}\to\Set$.
Moreover, the following diagram commutes:
\[
\xymatrix{
E\times_X Y\ar[rr]^{\pi_1}&&E\\
Y\ar[rr]_g\ar[u]^{g^*(s)}&&X\ar[u]_{s}
}
\]
We note that if $\pi$ has enough sections then so does the pullback $g^*(\pi)$.

\paragraph{Bundles with linear structure.}

Let $\pi:E\to X$ be a bundle. By a \emph{linear structure} on the bundle will be meant a structure of vector space on each fiber such that the operations of scalar multiplication and vector addition are globally continuous when regarded as maps $\CC\times E\to E$ and $E\times_X E\to E$, respectively, and such that the \emph{zero section} of $\pi$, which sends each $x\in X$ to $0_x$ (the zero of $E_x$), is continuous. Hence, a bundle equipped with a linear structure is a very loose form of vector bundle. We shall refer to bundles with such linear structures as \emph{linear bundles}.
The set $\sections(\pi)$ of continuous sections of a linear bundle $\pi$ is a vector space whose operations are computed fiberwise.

The \emph{category of linear bundles} over $X$, $\LBun(X)$, has the linear bundles as objects and, given linear bundles $\pi:E\to X$ and $\rho:F\to X$, a morphism
\[
h:\rho\to\pi
\]
is a map $h:\rho\to \pi$ in $\Bun(X)$ which is fiberwise linear: for each $x\in X$, $h$ restricts to a linear map $F_x\to E_x$. Hence, the induced mapping on sections $h_*:\sections(\rho)\to\sections(\pi)$
is linear, so we obtain a functor $(~)_*:\LBun\to\Vect$.
In addition, if $g:Y\to X$ is a continuous map the pullback $g^*(\pi)$ is a linear bundle, and we obtain a contravariant functor $(~)^*:\opp{(\Top/X)}\to\Vect$.

A \emph{trivial vector bundle} is a linear bundle of the form $\pi_2:A\times X\to X$ for some topological vector space $A$, with the obvious algebraic structure, and a \emph{locally trivial vector bundle} is a linear bundle $\pi:E\to X$ for which each $x\in X$ has an open neighborhood $U$ such that the restricted bundle $\pi_U$ is isomorphic to the trivial vector bundle $\pi_2:E_x\times U\to U$.
The pullback $g^*(\pi)$ of a locally trivial vector bundle $\pi:E\to X$ along $g:Y\to X$ is itself locally trivial.

The following simple fact will be useful later on.

\begin{lemma}\label{lemma:loctriv}
For any locally trivial vector bundle $\pi:E\to X$ with Hausdorff fibers, the image of the zero section is a closed set of $E$.
\end{lemma}

\begin{proof}
Let $v\in E$, $v\neq 0$, 
and let $U$ be a neighborhood of $\pi(v)\in X$
such that $E_U$ is isomorphic to a trivial vector bundle. Let $f: V\times U\to E_U$ be an isomorphism in $\LBun(U)$. Then $V$ is a Hausdorff space and
$f\bigl((V\setminus\{0\})\times U\bigr)\subset E\setminus\{0\}$ is an open set in $E$ containing $v$.
\qedhere
\end{proof}

\section{Quotient vector bundles}

Now we introduce the central notion of vector bundle in this paper.

\paragraph{Basic definitions and properties.}

The condition that a bundle \[\pi:E\to X\] has enough sections is equivalent to the requirement that the \emph{evaluation mapping}
\[\xymatrix{
\sections(\pi)\times X\ar[rrr]_{(s,x)\mapsto s(x)}^{\eval}&&&E
}
\]
be surjective. This fact suggests the following definition:

By a \emph{quotient bundle} will be meant a triple $(\pi,A,q)$ consisting of a bundle $\pi:E\to X$, a topological space $A$ and a map $q$ of bundles over $X$,
\[
\xymatrix{
A\times X\ar[rr]^{q}\ar[dr]_{\pi_2} && E\ar[dl]^{\pi}\\
&X
}
\]
referred to as the \emph{quotient map},
which is both surjective and open. This makes $E$ homeomorphic to a topological quotient of $A\times X$, hence the terminology. For each $x\in X$ we also write
\[
q_x:A\to E_x
\]
for the map defined by $q_x(a)=q(a,x)$.

\begin{lemma}\label{lem:subspaceeqquotient}
Let $(\pi:E\to X,A,q)$ be a quotient bundle. The following conditions hold:
\begin{enumerate}
\item $\pi$ is an open map.
\item The topology of each fiber $E_x$ as a subspace of $E$ coincides with the topology of $E_x$ regarded as a quotient of $A$.
\item The maps $q_x$ are continuous and open surjections.
\end{enumerate}
\end{lemma}

\begin{proof}
If $(\pi,A,q)$ is a quotient bundle then $q$ is a morphism in $\Bun(X)$ from $\pi_2:A\times X\to X$ to $\pi$, and thus the openness of $\pi$ is a consequence of the openness of $\pi_2$ plus the continuity and surjectivity of $q$: if $U\subset E$ is open then
\[
\pi(U)=\pi(q(q^{-1}(U)))=\pi_2(q^{-1}(U))\;.
\]
The second and third properties follow from a simple property of open maps: if $f:Y\to Z$ is an open map then for every subset $S\subset Z$ the restriction $f\vert_{f^{-1}(S)}:f^{-1}(S)\to S$ is itself an open map if $S$ and $f^{-1}(S)$ are equipped with their respective subspace topologies, because for every open set $U\subset Y$ we have $f(f^{-1}(S)\cap U)=S\cap f(U)$.
\qedhere
\end{proof}

\paragraph{Pullbacks.} Quotient bundles are well behaved under pullbacks:

\begin{lemma}\label{qbpbs}
The class of quotient bundles is closed under pullbacks along continuous base maps.
\end{lemma}

\begin{proof}
Let $(\pi:E\to X,A,q)$ be a quotient bundle, and $g:Y\to X$ a continuous map. The universal property of the pullback $g^*(\pi)$ ensures that there is a unique continuous map $q'$ that makes the following diagram commute:
\[
\xymatrix{
A\times Y\ar@{.>}[rd]|{q'}\ar[r]^-{\ident\times g}\ar@/_/[rdd]_{\pi_2}&A\times X\ar[dr]^q\\
&E\times_X Y\ar[d]^{g^*(\pi)}\ar[r]_-{\pi_1}&E\ar[d]^{\pi}\\
&Y\ar[r]_g&X
}
\]
The map $q'$ is defined by $q'(a,y)=(q(a,g(y)),y)$ for all $y\in Y$ and $a\in A$, and thus it is surjective. Moreover, the outer rectangle
\[
\xymatrix{
A\times Y\ar[r]^{\ident\times g}\ar[d]_{\pi_2}&A\times X\ar[d]^{\pi\circ q}\\
Y\ar[r]_g&X
}
\]
is itself a pullback diagram, and thus so is the upper rectangle
\[
\xymatrix{
A\times Y\ar[r]^{\ident\times g}\ar[d]_{q'}&A\times X\ar[d]^{q}\\
E\times_X Y\ar[r]_-{\pi_1}&\ E\;.
}
\]
Hence, $q'$ is open because it is the pullback of the open map $q$ along the continuous map $\pi_1$, and therefore $(g^*(\pi),A,q')$ is a quotient bundle. \qedhere
\end{proof}

\paragraph{Continuous sections.}

Let $(\pi:E\to X, A,q)$ be a quotient bundle.
For each $a\in A$ we define the continuous section
\[
\hat a:X\to E
\]
by $\hat a(x) = q(a,x)$. We denote the set of such sections by $\hat A$, and regard it as a topological space whose topology is the quotient topology obtained from $A$. Then, for each pair of open sets $U\subset X$ and $\Gamma\subset A$ we have the following subset of $E$:
\[
E_{\Gamma,U} = \bigcup_{a\in\Gamma} \hat a(U)\;.
\]

\begin{lemma}\label{qlbreflection}
Let $(\pi:E\to X,A,q)$ be a quotient bundle. The following conditions hold:
\begin{enumerate}
\item The family of subsets $E_{\Gamma,U}$ as defined above is a basis for the topology of $E$.
\item $(\pi,\hat A,\eval)$ is a quotient bundle.
\end{enumerate}
\end{lemma}

\begin{proof}
The sets $E_{\Gamma,U}$ form a basis for the topology of $E$  because $E_{\Gamma,U}=q(\Gamma\times U)$ and $q$ is open. For the second property
let $\chi:A\to \hat A$ be the quotient $a\mapsto\hat a$, and let $U\subset \hat A\times X$ be open. Then, denoting by $\phi$ the surjective map $\chi\times\ident$, we have
\[
\eval(U)=q(\phi^{-1}(U))\;,
\]
and thus $\eval:\hat A\times X\to E$ is open. \qedhere
\end{proof}

If $A=\hat A$ and $q=\eval$ we say that the quotient bundle is \emph{sectional}, and we denote it simply as the pair $(\pi,A)$.

\paragraph{Quotient bundles with linear structure.}

By a \emph{quotient vector bundle} will be meant a quotient bundle $(\pi:E\to X,A,q)$ such that $\pi$ is a linear bundle, $A$ is a topological vector space, and the open surjection $q:A\times X\to E$ is fiberwise linear.

\begin{example}\label{exm:banachasquotient}
As we shall see below, every Banach bundle on a locally compact Hausdorff space can be made a quotient vector bundle (\cf\ Theorem~\ref{banachasquotient}).
\end{example}

Every quotient vector bundle $(\pi,A,q)$ determines a \emph{kernel map}
\[
\kappa: X\to \Sub A
\]
defined by $\kappa(x) = q^{-1}(0_x)=\ker q_x$ for each $x\in X$.
We denote by $\mathfrak B(\pi,A,q)$ the intersection $\bigcap_{x\in X}\kappa(x)$, and call it the \emph{bundle radical} of $(\pi,A,q)$. Clearly, the quotient
$A\to\hat A$
defined by $a\mapsto\hat a$ is linear, and its kernel is $\mathfrak B(\pi,A,q)$. Hence, we have an isomorphism $\hat A\cong A/\mathfrak B(\pi,A,q)$, and the bundle is isomorphic to a sectional quotient vector bundle if and only if its bundle radical is $\{0\}$.

The following simple property will be useful later:

\begin{lemma}\label{lemmaJP-linindep}
Let $(\pi:E\to X,A,q)$ be a quotient vector bundle with kernel map $\kappa$, let \[a_1,\dots,a_n\in A\] be linearly independent, and let 
$V\subset A$ be the linear subspace spanned by $a_1,\ldots,a_n$.
If $x\in X$ and $\kappa(x)\cap V=\{0\}$ then $\hat a_1(x),\dots,\hat a_n(x)$ are linearly independent vectors of $E_x$.
\end{lemma}

\begin{proof}
There is a linear isomorphism $E_x\cong A/\kappa(x)$. Hence, if $\sum z_i\hat a_i(x)=0$ in $A/\kappa(x)$ for some $z_1,\dots,z_n\in\mathbb C$,
we have $\sum z_i a_i\in V\cap \kappa(x)$, and thus $z_1=\dots=z_n=0$. \qedhere
\end{proof}

\paragraph{Construction of quotient vector bundles.}

Let $A$ be a topological vector space, $X$ a topological space, and \[\kappa:X\to \Sub A\] an arbitrary map. Define $E$ to be the quotient space of $A\times X$ by the equivalence relation given, for all $a,b\in A$ and $x,y\in X$, by 
\[(a,x)\sim(b,y) \iff x=y \textrm{ and } a-b\in \kappa(x)\;.\]
Let us fix some notation and terminology.
\begin{enumerate}
\item We shall write $[a,x]$ for the equivalence class of $(a,x)\in A\times X$.
\item The \emph{quotient map induced by $\kappa$} is the map $q:A\times X\to E$ given by $q(a,x)=[a,x]$ for each $(a,x)\in A\times X$.
\item The \emph{bundle induced by $\kappa$} is the continuous surjection $\pi:E\to X$ given by the universal property of $q$; that is, for each $(a,x)\in A\times X$ we have
$\pi([a,x])=x$.
\end{enumerate}

\begin{theorem}\label{lem:qlb0}
Let $A$ be a topological vector space, $X$ a topological space, and $\kappa:X\to\Sub A$ a map. Let also $\pi:E\to X$ and $q:A\times X\to E$ be the bundle and quotient map induced by $\kappa$.
The following conditions are equivalent:
\begin{enumerate}
\item\label{qisopenmap} $q$ is an open map;
\item\label{piisquovb} $(\pi:E\to X,A,q)$ is a quotient vector bundle with kernel map $\kappa$.
\end{enumerate}
\end{theorem}

\begin{proof}
The implication $\eqref{piisquovb}\Rightarrow \eqref{qisopenmap}$ is true by definition of quotient vector bundle.
For the implication $\eqref{qisopenmap}\Rightarrow \eqref{piisquovb}$ we begin by noticing that
the continuity of $\pi$ is due to the universal property of $q$ as a quotient map, and all that we have left to do is prove that $\pi$ is a linear bundle. The fact that $q$ is open also implies that each fiber $E_x$ is, as a subspace of $E$, homeomorphic to the quotient $A/\kappa(x)$ (\cf\ Lemma~\ref{lem:subspaceeqquotient}). Hence, the fibers of $\pi$ are topological vector spaces, and we are left with proving that the linear operations are globally continuous on $E$. For addition we begin by noting that
\[
q\times q:A\times X\times A\times X\to E\times E
\]
is an open map, and thus, since $(A\times X)\times_X(A\times X)=(q\times q)^{-1}(E\times_X E)$, so is the restriction
\[
\widetilde q=q\times q\vert_{(A\times X)\times_X (A\times X)}:(A\times X)\times_X (A\times X)\to E\times_X E\;,
\]
which therefore is a quotient map.
The addition on $E$ is defined fiberwise by
\[
[a,x]+[b,x]=[a+b,x]\;,
\]
and thus it is the continuous map obtained from the addition on $A\times X$ via the universal property of $\widetilde q$ as a quotient map:
\[
\xymatrix{
(A\times X)\times_X(A\times X)\cong A\times A\times X\ar[rrr]^-{+}\ar[d]_{\widetilde q}&&&A\times X\ar[d]^q\\
E\times_X E\ar@{.>}[rrr]^-{+}&&& E
}
\]
Analogously, $\ident\times q:\CC\times A\times X\to\CC\times E$ is an open map, and thus scalar multiplication on $E$ is well defined as a continuous map on $\CC\times E$:
\[
\xymatrix{
\CC\times A\times X\ar[rrr]^-{m\times\ident}\ar[d]_{\ident\times q}&&& A\times X\ar[d]^q\\
\CC\times E\ar@{.>}[rrr]&&&E
}
\]
To conclude, the zero section is obviously continuous because for each $x\in X$ the sets $E_{\Gamma,U}$, with $\Gamma$ an open neighborhood of $0$ in $A$, form a local basis around $0_x$, and the preimage of each such set is the open set $U$. \qedhere
\end{proof}

In addition, we note the following:

\begin{corollary}\label{MaxvsT1}
If the equivalent conditions of Theorem~\ref{lem:qlb0} hold, the following assertions are equivalent:
\begin{enumerate}
\item $\kappa$ is valued in $\Max A$;
\item The fibers $E_x$ are Hausdorff spaces.
\end{enumerate}
\end{corollary}

\begin{proof}
The map $\kappa$ is valued in $\Max A$ if and only if the kernels $\ker q_x$ are closed, which is equivalent to the singletons $\{0_x\}$ being closed in the fibers $E_x$, and in turn is equivalent to all the singletons $\{e\}\subset E_x$ being closed. This means $E_x$ is a $T_1$ topological vector space, hence Hausdorff. \qedhere
\end{proof}

\paragraph{Hausdorff bundles.} Given a quotient vector bundle $(\pi:E\to X,A,q)$, a natural question regards the kind of topology which is carried by $E$. In particular, in general $E$ should not be expected to be a Hausdorff space even if $A$ and $X$ are, but the property of being Hausdorff is closely related to the zero section of the bundle:

\begin{theorem}\label{hausdorffbundles}
Let $(\pi:E\to X,A,q)$ be a quotient vector bundle.
If $E$ is Hausdorff the image of the zero section is a closed set in $E$. In addition, if both $A$ and $X$ are Hausdorff spaces then $E$ is Hausdorff if and only if the image of the zero section is closed.
\end{theorem}

Before we begin the proof, notice that the image of the zero section of $E$ is closed if and only if $q^{-1}(0)$
is closed in $A\times X$. Furthermore, $(a,x)\in q^{-1}(0)$ is equivalent to $a\in \kappa(x)$.

\begin{proof}
Assume first that $E$ is Hausdorff. We will show that $q^{-1}(0)$ is closed, which is equivalent to the image of the zero section being closed.
Let $(a,x)\notin q^{-1}(0)$. Then $q(a,x)\neq q(0,x)$, whence there are 
disjoint open sets $\mathcal U_a,\mathcal U_0\subset E$ such that
$q(a,x)\in \mathcal U_a$ and $q(0,x)\in \mathcal U_0$.
It follows that there are open sets
$W_0,W_a\subset A$ and $U\subset X$ such that $a\in W_a$, $0\in W_0$, $x\in U$,
$q(W_0\times U)\subset \mathcal U_0$ and 
$q(W_a\times U)\subset \mathcal U_a$.
Therefore we have
\[
q\bigl((W_a\times U)\cap q^{-1}(0)\bigr)\subset q(W_a\times U)\cap q(W_0\times U)\subset \mathcal U_a\cap\mathcal U_0=\emptyset\;.
\]
Hence, $(W_a\times U)\cap q^{-1}(0)=\emptyset$ and thus $q^{-1}(0)$ is closed.

Now assume that $q^{-1}(0)$ is closed and that both $A$ and $X$ are Hausdorff. We will show that $E$ is Hausdorff.
Let $(a,x),(b,y)\in A\times X$ be such that $q(a,x)\neq q(b,y)$.
We may assume that $b=0$ and,
since $X$ is Hausdorff, that $x=y$. 
Then $(a,x)\notin q^{-1}(0)$ and, since $q^{-1}(0)$ is closed, 
there is a neighborhood $W$ of $a$
and a neighborhood $U$ of $x$ such that $(W\times U)\cap q^{-1}(0)=\emptyset$.
Then, for all $y\in U$ we have $W\cap \kappa(y)=\emptyset$.
Let $W'$ be a neighborhood of the origin such that $a+W'-W'\subset W$.
For all $b\in W'$ and $c\in a+W'$ we have $c-b\in a+W'-W'\subset W$, and thus $c-b\notin \kappa(y)$, which means that $q(b,y)\neq q(c,y)$. Hence, $q(W'\times U)$ and $q\bigl((a+W')\times U\bigr)$
are disjoint open sets, and it follows that $E$ is Hausdorff.  \qedhere
\end{proof}

\begin{corollary}
Let $(\pi:E\to X,A,q)$ be a quotient vector bundle with Hausdorff fibers such that $\pi$ is locally trivial and both $X$ and $A$ are Hausdorff spaces. Then $E$ is a Hausdorff space.
\end{corollary}

\begin{proof}
Immediate consequence of Lemma~\ref{lemma:loctriv} and Theorem~\ref{hausdorffbundles}. \qedhere
\end{proof}

\paragraph{Normed bundles.}

We shall say that a quotient vector bundle \[(\pi:E\to X,A,q)\] is \emph{normed} if $A$ is a normed linear space. Then we define a mapping
\[
{\norm~} : E\to\RR\;,
\]
called the \emph{quotient norm} on $E$, by, for each $x\in X$ and $a\in A$,
\[
\norm{q(a,x)} = d(a,\kappa(x)) = \inf_{p\in \kappa(x)}\norm{ a+p}\;,
\]
where $\kappa$ is the kernel map of the bundle.

\begin{theorem}\label{thm:normedqvb}
Let $(\pi:E\to X,A,q)$ be a normed quotient vector bundle. The quotient norm on $E$ is upper semicontinuous.
\end{theorem}

\begin{proof}
Let $\varepsilon>0$, and let $\Gamma=B_{\varepsilon}(0)\subset A$. The basic open set $E_{\Gamma,X}$ consists of all the elements $q(a,x)$ with $\norm a<\varepsilon$, so we have
\[
\norm{ q(a,x)}\le\norm a<\varepsilon\;.
\]
Hence, $E_{\Gamma,X}\subset\{e\in E\st \norm e<\varepsilon\}$. Now let $e=q(a,x)\in E$ be such that $\norm e<\varepsilon$. The condition $\norm{d(a,\kappa(x))}<\varepsilon$ implies that $\norm{a+p}<\varepsilon$ for some $p\in\kappa(x)$. Then $a+p\in\Gamma$ and $q(a+p,x)=e$, showing that $E_{\Gamma,X}=\{e\in E\st \norm e<\varepsilon\}$. Hence, $\{e\in E\st \norm e<\varepsilon\}$ is open for arbitrary $\varepsilon>0$, which means precisely that ${\norm~}:E\to\RR$ is upper-semicontinuous. \qedhere
\end{proof}

We shall say that the normed quotient vector bundle $(\pi,A,q)$ is \emph{continuous} if moreover the norm of $E$ is also lower-semicontinuous.

\section{Banach bundles}\label{sec:banach}

In this section we shall see that classical Banach bundles can always be regarded as continuous quotient vector bundles, at least if the base space is locally compact Hausdorff.

\paragraph{Basic definitions and facts.}

Let $X$ be a Hausdorff space. By a \emph{Banach bundle} over $X$ \cite{FD1}*{II.13.4} is meant a Hausdorff space $E$ equipped with a continuous open surjection
\[\pi:E\to X\]
such that:
\begin{enumerate}
\item for each $x\in X$ the fiber $E_x$ has the structure of a Banach space;
\item\label{banachtopaddition} addition is continuous on $E\times_X E$ to $E$;
\item for each $\lambda\in\CC$, scalar multiplication $e\mapsto\lambda e$ is continuous on $E$ to $E$;
\item\label{banachtopnorm} $e\mapsto \norm e$ is continuous on $E$ to $\RR$;
\item\label{banachtopbasis} for each $x\in X$ and each open set $V\subset E$ containing $0_x$, there is $\varepsilon>0$ and an open set $U\subset X$ containing $x$ such that $E_U\cap T_\varepsilon\subset V$, where $T_\varepsilon$ is the ``tube'' $\{e\in E\st \norm e<\varepsilon\}$.
\end{enumerate}

Condition \eqref{banachtopbasis} is equivalent to stating that for each $x\in X$ the open ``rectangles''
\[E_U\cap T_\varepsilon\]
with $x\in U$ form a local basis of $0_x$. It is also equivalent to the statement that for every net $(e_\alpha)$ in $E$, if $\pi(e_\alpha)\to x$ and $\norm{e_\alpha}\to 0$ then $e_\alpha\to 0_x$ (the axiom of choice is needed for the converse implication).

Hence, in particular, the zero section of a Banach bundle is continuous. It can also be shown that scalar multiplication as an operation $\CC\times E\to E$ is continuous, and thus Banach bundles are linear bundles.

We further recall that if $X$ is locally compact then any Banach bundle over $X$ has enough sections \cite{FD1}*{Appendix C}.

\paragraph{Banach bundles with enough sections.}

Let $\pi:E\to X$ be a Banach bundle with enough sections. For each $s\in\sections(\pi)$ and each $\varepsilon>0$ define the set
\[
T_{\varepsilon}(s)=\{e\in E\st \norm{e-s(\pi(e))}<\varepsilon\}\;.
\]

\begin{lemma}
Let $\pi:E\to X$ be a Banach bundle with enough sections. The collection of all the sets of the form
\[
E_U\cap T_\varepsilon(s)\;,
\]
where $s\in\sections(\pi)$ and $U\subset X$ is open, is a basis for the topology of $E$.
\end{lemma}

\begin{proof}
$T_\varepsilon(s)$ is the image of $T_\varepsilon$ by the homeomorphism
\[
h_s:E\to E
\]
which is defined by $h_s(e)=e+s(\pi(e))$ (whose inverse $h_s^{-1}$ is $h_{-s}$), and thus it is an open set of $E$. Then, for each $e\in E$ and each continuous section $s$ through $e$, the collection
\[
\{E_U\cap T_{\varepsilon}(s)\st \varepsilon>0,\ U\subset X\textrm{ is open,}\ \pi(e)\in U\}
\]
is a local basis at $e$. \qedhere
\end{proof}

\paragraph{Banach bundles on locally compact spaces.}

Now we shall see that Banach bundles on locally compact Hausdorff spaces yield quotient vector bundles (\cf\ Example~\ref{exm:banachasquotient}).

\begin{theorem}\label{banachasquotient}
Every Banach bundle on a locally compact Hausdorff space can be made a continuous normed sectional quotient vector bundle whose quotient norm coincides with the Banach bundle norm.
\end{theorem}

\begin{proof}
Let $\pi:E\to X$ be a Banach bundle with $X$ locally compact Hausdorff, and denote by $\sections_0(\pi)$ the Banach space of continuous sections of $\pi$ which vanish at infinity, equipped with the topology of the supremum norm $\norm~_\infty$. The evaluation mapping
\[
q=\eval:\sections_0(\pi)\times X\to E
\]
is surjective because there are enough sections due to the local compactness of $X$, and thus there are enough sections in $\sections_0(\pi)$ because there are enough compactly supported sections (from any continuous section $s$ through $e$ and any compactly supported continuous function $f:X\to\CC$ such that $f(\pi(e))=1$ we obtain a compactly supported continuous section $fs$ through $e$). The evaluation mapping is also continuous because the supremum norm topology contains the compact-open topology. For each $x\in X$ the mapping
\begin{eqnarray*}
\sections_0(\pi)&\to& E_x\\
s&\mapsto& q(s,x)
\end{eqnarray*}
is linear, and thus $q$ is a map in $\LBun(X)$. Now let us prove that $q$ is an open map. A basis for the topology of $\sections_0(\pi)\times X$ consists of all the sets $B_\varepsilon(s)\times U$ with $s\in \sections_0(\pi)$, $\varepsilon>0$, and $U\subset X$ open. The image $q(B_\varepsilon(s)\times U)$ is of course contained in $T_\varepsilon(s)\cap E_U$. For the converse inclusion let $e\in T_\varepsilon(s)\cap E_U$, and let $t\in \sections_0(\pi)$ be such that $t(\pi(e))=e$. We have $\norm{t(\pi(e))-s(\pi(e))}<\varepsilon$, and thus, by the upper semi-continuity of the norm of $E$, for some open set $V\subset U$ containing $\pi(e)$ we have $\norm{t(x)-s(x)}<\varepsilon$ for all $x\in V$. Let $f:X\to [0,1]$ be a continuous compactly supported function such that both $f(\pi(e))=1$ and $\supp f\subset V$, and let $t'=s+f(t-s)$. Then $t'\in \sections_0(\pi)$ and $t'(\pi(e))=e$, and furthermore we obtain
\[
\norm{t'-s}_\infty=\norm{f(t-s)}_\infty\le\max_{x\in\supp f} \{\norm{t(x)-s(x)}\}< \varepsilon\;,
\]
so we conclude that $e\in q(B_\varepsilon(s)\times U)$. Hence, $q(B_\varepsilon(s)\times U)=T_\varepsilon(s)\cap E_U$, and it follows that $q$ is an open map.
So we have a normed sectional quotient vector bundle $(\pi,\sections_0(\pi),\eval)$. The  quotient norm $\norm~_q$ on $E$ is defined for all $s\in \sections_0(\pi)$ and $x\in X$ by (\cf\ definition preceding Theorem~\ref{thm:normedqvb})
\[
\norm{s(x)}_q=\inf\bigl\{\norm{s+p}_\infty\st p\in \sections_0(\pi)\textrm{ and }p(x)=0_x\bigr\}\;.
\]
Let us show that $\norm~_q$ coincides with the Banach bundle norm $\norm~$. First, for all $x\in X$ and all $s,p\in\sections_0(\pi)$ such that $p(x)=0_x$ we have
\[
\norm{s(x)}=\norm{s(x)+p(x)}\le\norm{s+p}_\infty\;,
\]
and thus $\norm~\le\norm~_q$. Now in order to prove that we have $\norm~=\norm~_q$ it suffices to show that for all $x\in X$ and $s\in \sections_0(\pi)$ there is $p\in \sections_0(\pi)$ such that $p(x)=0_x$ and $\norm{s+p}_\infty=\norm{s(x)}$. There are two cases: if $s(x)=0_x$ just let $p=-s$; otherwise let $p=fs$ where $f:X\to(-1,0]$ is the continuous function defined by, for all $y\in X$,
\[
f(y)=\left\{\begin{array}{ll}
\frac{\norm{s(x)}}{\norm{s(y)}}-1&\textrm{if }\norm{s(x)}\le\norm{s(y)}\;,\\
0&\textrm{if }\norm{s(x)}\ge\norm{s(y)}\;.\end{array}\right.
\]
Finally, the definition of Banach bundle states that the norm on $E$ is also lower-semicontinuous, so it follows that the normed quotient vector bundle $(\pi,\sections_0(\pi),\eval)$ is continuous.
\qedhere
\end{proof}

\section{Classifying spaces}\label{sec:classif}

Now we shall look at conditions on a map $\kappa:X\to\Sub A$ which ensure that it is the kernel map of a quotient vector bundle. As we shall see, this is so when $\kappa$ is continuous with respect to a suitable topology on $\Sub A$.

\paragraph{Lower Vietoris topology.}

The Vietoris topology on the space of closed subsets of a topological space \cite{Vietoris} is often presented as the coarsest topology that contains both the lower and the upper Vietoris topologies --- see, \eg, \cite{NT96}. For the purposes of this section, given a topological vector space $A$ we shall topologize $\Max A$ (and indeed also $\Sub A$) with the subspace topology which is obtained from the lower Vietoris topology.

Let $A$ be a topological vector space. For each open set $U\subset A$ we shall write
\[
\ps U=\{P\in\Sub A\st U\cap P\neq\emptyset\}\;.
\]
The collection of all the sets $\ps U$ is a subbasis for a topology on $\Sub A$, which we shall refer to as the \emph{lower Vietoris topology on $\Sub A$}.

We remark that $\Sub A$ with this topology is usually not a $T_0$ space, since for all $V\in\Sub A$ the neighborhoods of a linear subspace $V$ are the same as the neighborhoods of its closure $\overline V$. However, the subset $\Max A$, equipped with the subspace topology, is always a $T_0$ space.

\begin{theorem}\label{basictoppropsMax}
Let $A$ be a topological vector space.
\begin{enumerate}
\item\label{spec1}
$\Max A$ is a topological retract of $\Sub A$.
\item\label{spec2} The map $\linspan{~}:A\to\Max A$ which to each $a\in A$ assigns its span $\linspan a=\CC a$ is continuous.
\end{enumerate}
\end{theorem}

\begin{proof}
The mapping $V\mapsto\overline V$ from $\Sub A$ to $\Max A$ is continuous, so \eqref{spec1} holds. In order to prove \eqref{spec2},
let $a\in A$ and let $W\subset\Max A$ be an open neighborhood of $\linspan a$. It suffices to take
$W=\ps U\cap\Max A$ for some open set $U\subset A$. Then there is $\lambda\in \CC$ such that $\lambda a\in U$. Due to the continuity of scalar multiplication in $A$ there is $V\subset A$ open such that $a\in V$ and $\lambda c\in U$ for all $c\in V$. Hence,
$\linspan c\in W$ for all $c\in V$, showing that \eqref{spec2} holds. \qedhere
\end{proof}

\paragraph{Classification of quotient vector bundles.}

Let again $A$ be a topological vector space, $X$ a topological space, and \[\kappa:X\to \Sub A\] an arbitrary map. Let $q:A\times X\to E$ be the quotient map induced by $\kappa$ (\cf\ definition above Theorem~\ref{lem:qlb0}), and, for each $Y\subset A\times X$,
write $[Y]$ for $q^{-1}\bigl(q(Y)\bigr)$ (the saturation of $Y$).

\begin{lemma}\label{lem:neighborigin}
The induced quotient map is open if and only if for every neighborhood $U\subset A$
of the origin the set $[U\times X]=\bigl\{(a,x)\in A\times X\st a\in \kappa(x)+U\bigr\}$
is an open subset of $A\times X$.
\end{lemma}
\begin{proof}
The forward implication is trivial. In order to prove the reverse implication let $W\subset X$ and $U\subset A$ be open sets. We want to show that
$[U\times W]=q^{-1}\bigl(q(U\times W)\bigr)$ is an open set. We easily see that
\[
[U\times W]=\bigcup_{x\in W}\bigl(\kappa(x)+U\bigr)\times \{x\}
=\bigl\{(a,x)\in A\times W \st a\in \kappa(x)+U\bigr\}
\]
Given any $v\in A$ the translation 
$\tau_v: A\times X\to A\times X$ given 
by $\tau_v(a,x)=(a+v,x)$ is a homeomorphism
and $\tau_v[U\times W]=[(U+v)\times W]$,
so $q(U\times W)$ is open if and only if
$q\bigl((U+v)\times W\bigr)$ is open, and thus we may assume that
$U$ is a neighborhood of the origin. Now let 
$(a,x)\in [U\times W]\subset [U\times X]$. By hypothesis
$[U\times X]$ is open, so there are open sets $W'\subset X$ and $U'\subset A$
such that $(a,x)\in U'\times W'\subset [U\times X]$. Hence, 
$U'\times (W'\cap W) \subset [U\times W]$. \qedhere
\end{proof}

\begin{theorem}
Let $\kappa:X\to\Sub A$ be a map. The induced quotient map is open if and only if $\kappa$ is continuous with respect to the lower Vietoris topology.
\end{theorem}

\begin{proof}
First we assume that $\kappa$ is continuous. In order to prove that the induced quotient map is open we shall use Lemma~\ref{lem:neighborigin}. Let $U\subset A$ be a neighborhood of the origin and
let $(a,x)\in[U\times X]$. Then 
$a\in \kappa(x)+U$ so we can write $a=a_0+a_1$ with $a_0\in \kappa(x)$ and $a_1\in U$.
Consider the map $f: A\times A\to A$ given by $f(v_1,v_2)=a_1+v_1-v_2$
and let $U'\subset A$ be a neighborhood of the origin so that $f(U'\times U')=a_1+U'-U'\subset U$.
Since $\kappa$ is continuous and $\kappa(x)\in\ps{a_0+U'}$,
there is a neighborhood $W$ of $x$ such that,
for any $y\in W$, we have $\kappa(y)\cap (a_0+U')\neq\emptyset$.
Then $(a+U')\times W$ is a neighborhood of $(a,x)$, so we only have to show that
$(a+U')\times W\subset[U\times X]$.
Let $(a',x')\in (a+U')\times W$. We need to show that
$a'\in \kappa(x')+U$. Since $x'\in W$, we have $\kappa(x')\cap(a_0+U')\neq\emptyset$ and hence
$a_0\in \kappa(x')-U'$. Since $a'\in a+U'=a_0+a_1+U'$,
we have $a'\in a_1+U'-U'+\kappa(x')\subset U+\kappa(x')$, which
concludes the first half of the proof.

Now assume that the induced quotient map is open. Let $x\in X$ and let $U\subset A$
be an open set such that $\kappa(x)\in\ps U$ (equivalently: 
$U\cap \kappa(x)\neq\emptyset$). 
We want to find a neighborhood $W\subset X$ of $x$ so that
$U\cap \kappa(y)\neq\emptyset$ for every $y\in W$.
Let $a\in U\cap \kappa(x)$ and consider the neighborhood $U'=a-U$ of the origin.
Since $a\in \kappa(x)$ and $0\in U'$, we have
$(a,x)\in[U'\times X]$
which is, by hypothesis, open, so there is a neighborhood $W\subset X$
of $x$ such that $\{a\}\times W\subset [U'\times X]$
and, for every $y\in W$, we obtain $a\in \kappa(y)+U'=a+\kappa(y)-U$. It follows
that $0\in \kappa(y)-U$, and therefore $U\cap \kappa(y)\neq\emptyset$. \qedhere
\end{proof}

\paragraph{Universal bundles.}

Let $A$ be a topological vector space. We shall refer to the quotient vector bundle
\[
(\pi_A:E_A\to\Sub A, A, q)
\]
whose kernel map is the identity map $\ident:\Sub A\to\Sub A$ as the \emph{universal quotient vector bundle for $A$}.

The following is an immediate corollary of the previous results:

\begin{theorem}
Every quotient vector bundle $(\pi:E\to X,A,q)$ is isomorphic to the pullback of the universal bundle $\pi_A$ along the continuous map
\[\kappa:X\to\Sub A\]
which is defined by $\kappa(x)=\ker q_x$.
Moreover, the fibers $E_x$ are Hausdorff spaces if and only if $\kappa$ is valued in $\Max A$.
\end{theorem}

\paragraph{Fell topology.}

Let $A$ be a topological vector space. For each compact set $K\subset A$ we define
\[\ft K = \{P\in\Max A\st P\cap K=\emptyset\}\;,\]
and the \emph{Fell topology} on $\Max A$ is the coarsest topology that contains the lower Vietoris topology and all the sets $\ft K$ \cite{Fell62} (see also \cite{NT96}). We shall use analogous notation for finite dimensional linear subspaces $V\subset A$, as follows:
\[
\ft V=\bigl\{P\in \Max A\st P\cap V=\{0\}\bigr\}\;.
\]
Note that if $K_1,K_2\subset A$ are compact and $K_3=K_1\cup K_2$ then $\ft K_1\cap\ft K_2=\ft K_3$, and
therefore a basis for the Fell topology consists of all the sets of the form
\[
\ps{U_1}\cap\ldots\cap\ps{U_m}\cap\ft K\;,
\]
where the $U_i$'s are open sets of $A$ and $K\subset A$ is compact.

\begin{lemma}\label{lemmaJP-UVopen}
Let $A$ be a Hausdorff vector space.
For any finite dimensional subspace
$V\subset A$ the set $\ft V$
is open in the Fell topology.
\end{lemma}

\begin{proof}
Let $V\subset A$ be a finite dimensional vector space, hence with the Euclidean topology, and let $K$ be the unit sphere in $V$ in some norm. Then $K$ is compact and
$\ft V=\ft K$.  \qedhere
\end{proof}

\paragraph{Bundles classified by the Fell topology.}

If $X$ is a topological space, we shall say that a map $\kappa:X\to\Max A$ is \emph{Fell-continuous} if it is continuous with respect to the Fell topology of $\Max A$. 
The Fell topology contains the lower Vietoris topology and therefore Fell-continuous maps $\kappa:X\to\Max A$ determine quotient vector bundles $(\pi,A,q)$.

\begin{lemma}
Let $(\pi:E\to X,A,q)$ be a quotient vector bundle with Fell-continuous kernel map $\kappa:X\to\Max A$, for some Hausdorff vector space $A$. For any $x\in X$ such that $E_x$ is finite dimensional there is a neighborhood $U$ of
$x$ such that $\dim E_y\geq \dim E_x$ for any $y\in U$.
\end{lemma}

\begin{proof}
Let $x\in X$ and $a_1,\ldots,a_n\in A$ be such that $\hat a_1(x),\ldots,\hat a_n(x)$ is a basis of $E_x$. Let also $V=\linspan{ a_1,\ldots,a_n}$ and $U=\kappa^{-1}(\ft V)$. Then $x\in U$ and, by Lemma~\ref{lemmaJP-UVopen}, $U$ is open. For any $y\in U$ we have $\kappa(y)\cap V=\{0\}$ and thus, by Lemma~\ref{lemmaJP-linindep}, $\hat a_1(y),\ldots,\hat a_n(y)$ are linearly independent in $E_y$.  \qedhere
\end{proof}

The following result provides a characterization of what it means for a bundle to have a Fell-continuous kernel map.

\begin{theorem}\label{fellcont}
Let $(\pi:E\to X,A,q)$ be a quotient vector bundle with kernel map $\kappa:X\to\Max A$.
If the image of the zero section of the bundle is closed then $\kappa$ is Fell-continuous. In addition, if both $X$ and $A$ are first countable, the image of the zero section is closed if and only if $\kappa$ is Fell-continuous.
\end{theorem}

\begin{proof}
Assume
that the image of the zero section is closed. Let
$K\subset A$ be a compact set and let $x\in\kappa^{-1}(\ft K)$ --- that is, 
$\kappa(x)\cap K=\emptyset$.
Then, for each $a\in K$ we have $(a,x)\notin q^{-1}(0)$.
By hypothesis $q^{-1}(0)$ is closed, and thus there are open sets $U_a\subset X$ and 
$W_a\subset A$ such that $(a,x)\in W_a\times U_a$ and 
$(W_a\times U_a)\cap q^{-1}(0)=\emptyset$. The collection $\{W_a\}$ forms
an open cover of $K$, whence there is a finite subcover $W_{a_1},\dots,W_{a_n}$. Let $U=\bigcap_i U_{a_i}$. Then $(K\times U)\cap q^{-1}(0)=\emptyset$. In other words, for every $y\in U$ we have $\kappa(y)\cap K=\emptyset$, and thus $x\in U\subset\kappa^{-1}(\ft K)$.
We showed that $\kappa$ is Fell-continuous.

Now assume that $\kappa$ is Fell-continuous and that $A$ and $X$ are both first countable.
Suppose $q^{-1}(0)$ is not closed. 
Then there is a converging sequence $(a_n,x_n)$
in $q^{-1}(0)$ with limit $(a,x)\notin q^{-1}(0)$. 
Since $a_n\to a$ and $a\notin \kappa(x)$, and $\kappa(x)$ is closed, there is
a positive integer $p$ such that $a_n\notin \kappa(x)$ for all $n>p$. 
Consider the compact set
$K=\{a_n\st n>p\}\cup\{a\}\subset A$. Then, by the continuity of $\kappa$, 
the set $U=\{y\in X\st \kappa(y)\cap K=\emptyset\}$ is open in $X$ and $x\in U$.
But $x_n\to x$ and
$a_n\in \kappa(x_n)$ for all $n$, and thus for some $n$ we have $(a_n,x_n)\notin q^{-1}(0)$, which is a contradiction. Therefore $q^{-1}(0)$ is closed in $A\times X$.  \qedhere
\end{proof}

We have thus obtained a new necessary condition for local triviality, in terms of Fell-continuity:

\begin{corollary}\label{loctrivimpliesfell}
Let $(\pi:E\to X,A,q)$ be a quotient vector bundle with kernel map $\kappa:X\to\Max A$, such that $\pi$ is locally trivial. Then $\kappa$ is Fell-continuous.
\end{corollary}

\begin{proof}
This is an immediate consequence of Lemma~\ref{lemma:loctriv}, Corollary~\ref{MaxvsT1}, and Theorem~\ref{fellcont}. \qedhere
\end{proof}

We also record some simplified consequences of Theorem~\ref{fellcont} when the spaces involved are all good enough:

\begin{corollary}\label{fellcorollary}
Let $(\pi:E\to X,A,q)$ be a quotient vector bundle with kernel map $\kappa:X\to\Max A$, such that both $X$ and $A$ are first countable Hausdorff spaces. The following are equivalent:
\begin{enumerate}
\item $\kappa$ is Fell-continuous.
\item The image of the zero section of the bundle is closed.
\item $E$ is a Hausdorff space.
\end{enumerate}
\end{corollary}

\begin{proof}
Immediate consequence of Theorem~\ref{hausdorffbundles} and Theorem~\ref{fellcont}. \qedhere
\end{proof}

\paragraph{The closed balls topology.}

For a normed vector space $A$ we shall consider a refinement of the Fell topology on $\Max A$ such that instead of defining open sets $\ft K$ for compact $K$ we shall consider instead a family of open sets indexed by $A\times\RR_{>0}$ as follows: for each $a\in A$ and $r>0$ we let
\[
\cb a r = \{P\in\Max A\st d(a,P)>r\}\;.
\]
The coarsest topology that contains the lower Vietoris topology and the open sets $\cb a r$ has a basis consisting of sets of the form
\[
\ps{U_1}\cap\ldots\cap\ps{U_m}\cap\cb{a_1}{r_1}\cap\ldots\cap \cb{a_k}{r_k}\;,
\]
where the $U_i$'s are open sets of $A$. We refer to this topology as the \emph{closed balls topology} on $\Max A$ because for all $P\in\cb a r$ we have $P\cap\overline{B_r(a)}=\emptyset$, and thus the definition of $\cb a r$ resembles that of $\ft K$ if we replace the compact set $K$ by the closed ball $B:=\overline{B_r(a)}$. Indeed, if $A$ is reflexive (for instance a Hilbert space) we have $\cb a r=\check B$, although more generally, for any normed space $A$, only the inclusion $\cb a r\subset\check B$ holds.

\begin{lemma}\label{lemmaJP-d(k(x),K)}
Let $A$ be a normed vector space. For every $\varepsilon >0$ and
every compact $K\subset A$, the set 
\[\ft K_\varepsilon=\{P\in \Max A\st d\bigl(P,K\bigr)>\varepsilon\bigr\}\] is open in the closed balls topology of $\Max A$.
\end{lemma}

\begin{proof}
Let $\varepsilon>0$, and fix $P\in \ft K_\varepsilon$. Choose $\varepsilon'$ and $\delta$ so that
$\varepsilon<\varepsilon'<d(P,K)$ and $\delta<\varepsilon'-\varepsilon$.
Cover $K$ by a finite number of balls $B_{\delta}(a_i)$ with $a_i\in K$ and let
$U=\bigcap_i\cb{a_i}{\varepsilon'}$. Then $U$ is open in the closed balls topology, and $P\in U$. 
We will show that $U\subset \ft K_\varepsilon$.
Let $Q\in U$. Given $u\in Q$ and $v\in K$, there is $i$ such that
$\norm{v-a_i}<\delta$. Then
\[
\varepsilon'<\norm{u-a_i}\leq\norm{u-v}+\norm{v-a_i}<\norm{u-v}+\delta\;,
\]
so $\norm{u-v}>\varepsilon'-\delta$. Since $u$ and $v$ are arbitrary we get 
$d(Q,K)\geq\varepsilon'-\delta>\varepsilon$, and thus $Q\in \ft K_\varepsilon$.
Hence, $\ft K_\varepsilon$ is open in the closed balls topology.  \qedhere
\end{proof}

\begin{lemma}\label{cbcontainsfell}
Let $A$ be a normed vector space. The closed balls topology of $\Max A$ contains the Fell topology.
\end{lemma}

\begin{proof}
Let $K\subset A$ be compact, and recall the definition of $\ft K_\varepsilon$ from Lemma~\ref{lemmaJP-d(k(x),K)}. Clearly,
\[
\bigcup_{\varepsilon >0} \ft K_\varepsilon\subset\ft K\;.
\]
In order to see that also the converse inclusion holds let $P\in\ft K$. The condition $P\cap K=\emptyset$ implies $d(P,k)>0$ for all $k\in K$ because $P$ is closed, and this further implies $d(P,K)>0$ because $d(P,-):K\to\RR$ is continuous and $K$ is compact. Hence, choosing $\varepsilon$ such that
$d(P,K)>\varepsilon>0$ we obtain $P\in\ft K_\varepsilon$, and thus
\[
\bigcup_{\varepsilon >0} \ft K_\varepsilon=\ft K\;. \qedhere
\]
\end{proof}

\paragraph{Normed bundles and Banach bundles.} The closed balls topology classifies continuous normed bundles, as we now show.

\begin{theorem}\label{classcontnorm}
Let $(\pi:E\to X,A,q)$ be a normed quotient vector bundle with kernel map $\kappa:X\to\Max A$. The following are equivalent:
\begin{enumerate}
\item\label{cln1} The bundle is continuous;
\item\label{cln2} $\kappa$ is continuous with respect to the closed balls topology.
\end{enumerate}
\end{theorem}

\begin{proof}
Let us prove $\eqref{cln1}\Rightarrow \eqref{cln2}$. Let $a\in A$ and $\varepsilon>0$.
The distance $d\bigl(a,\kappa(x)\bigr)$ equals $\norm{ q(a,x)}$, so we have
\begin{eqnarray*}
\kappa^{-1}\bigl(\cb a\varepsilon\bigr)&=&\bigl\{x\in X\st d(a,\kappa(x))>\varepsilon\bigr\}\\
&=&\bigl\{x\in X\st \norm{ \hat a(x)}>\varepsilon\bigr\}\\
&=&
{\hat a}^{-1}\bigl(\{e\in E \st \norm e>\varepsilon\}\bigr)\;,
\end{eqnarray*}
and thus $\kappa^{-1}\bigl(\cb a\varepsilon\bigr)$ is open due to the continuity of the norm on $E$ and the continuity of $\hat a$, and we conclude that $\kappa$ is continuous with respect to the closed balls topology.

Now let us prove $\eqref{cln2}\Rightarrow \eqref{cln1}$. We only need to check that the norm on $E$ is lower semicontinuous; that is,
given $a\in A$, $x\in X$, and $r>0$ such that $\norm{q(a,x)}>r$, we need to prove that for some neighborhood $W$ of $q(a,x)$ we have $\norm e>r$ for all $e\in W$, where $W$ can be taken to be $q(B_\varepsilon(a)\times U)$ for some $\varepsilon>0$ and some open neighborhood $U$ of $x$. Let $r_0=\norm{q(a,x)}$ and let $m=(r_0+r)/2$. Note that for all $y\in X$ the condition $\norm{q(a,y)}>m$ is equivalent to the statement that $y\in U$ for the open set defined by
\[
U=\kappa^{-1}(\cb a m)\;.
\]
In particular, $x\in U$. Let $\varepsilon=(r_0-r)/2$. Then for all $b\in B_\varepsilon(a)$ and all $y\in U$ we have
\begin{eqnarray*}
\norm{q(b,y)} &=& d(b,\kappa(y))\\
&\ge& d(a,\kappa(y)) - d(a,b)\\
&=&\norm{q(a,y)} - \norm{a-b}\\
&>&m-\varepsilon=r\;. \qedhere
\end{eqnarray*}
\end{proof}

\begin{corollary}\label{classbanach}
Let $A$ be a Banach space, $X$ a first countable Hausdorff space, and $(\pi:E\to X,A,q)$ a quotient vector bundle with kernel map $\kappa:X\to\Max A$. The following are equivalent:
\begin{enumerate}
\item\label{cb1} The quotient norm makes $\pi:E\to X$ is a Banach bundle;
\item\label{cb2} $\kappa$ is continuous with respect to the closed balls topology.
\end{enumerate}
\end{corollary}

\begin{proof}
$\eqref{cb1}\Rightarrow \eqref{cb2}$ is an immediate consequence of Theorem~\ref{classcontnorm} because Banach bundles have continuous norm. Conversely, let us prove $\eqref{cb2}\Rightarrow \eqref{cb1}$. If \eqref{cb2} holds then the quotient norm is continuous, again by Theorem~\ref{classcontnorm}. In addition, $\kappa$ is Fell-continuous by Lemma~\ref{cbcontainsfell}, so by Corollary~\ref{fellcorollary} the space $E$ is Hausdorff. To conclude, the fibers $E_x$ are Banach spaces because they are isomorphic to quotients $A/\kappa(x)$ with $\kappa(x)\in\Max A$ and, by Lemma~\ref{qlbreflection}, the topology of $E$ around the image of the zero section is the required one (\cf\ axiom \ref{banachtopbasis} in the definition of section \ref{sec:banach}). \qedhere
\end{proof}

\section{Finite rank and local triviality}

Let $(\pi:E\to X,A,q)$ be a quotient vector bundle and $\kappa:X\to\Sub A$ its kernel map. We say the bundle has \emph{rank $n$} if all of its fibers $E_x$ have dimension $n$ or, equivalently, if all the subspaces $\kappa(x)$ have codimension $n$ in $A$. 
In this section we study such bundles. We shall only be interested in bundles whose fibers have the Euclidean topology, so we shall take $\kappa$ to be valued in $\Max A$, in addition assuming that $\kappa$ is Fell-continuous
because one of our aims is to study locally trivial bundles (\cf\ Corollary~\ref{loctrivimpliesfell}).
From here on we shall denote by \[\Max_n A\] the topological space, equipped with the relative Fell topology, whose points are the closed linear subspaces $P\in\Max A$ such that $\dim(A/P)=n$.

\paragraph{Fiber structures.}

Let $A$ be a Hausdorff vector space and let \[(\pi:E\to X,A,q)\] be a rank-$n$ quotient vector bundle with Fell-continuous kernel map \[\kappa:X\to\Max_n A\;.\] Given any $n$-dimensional subspace $V\subset A$ and any $x\in X$, both $V$ and $E_x$ have the Euclidean topology, so there is a linear homeomorphism $E_x\cong V$. This suggests that, in a suitable sense, $V$ can be regarded as being ``the fiber'' of the bundle.

In order to pursue this idea, first note that for each $P\in\ft V\cap\Max_n A$ we have $V\cong A/P$ and $P\cap V=\{0\}$, and therefore $A=V\oplus P$, where for each $a\in A$ the component of $a$ in $V$ is the vector $v$ such that $V\cap(a+P)=\{v\}$ (\cf\ Lemma~\ref{singleton}). This leads to the following definitions:

\begin{enumerate}
\item
The \emph{fiber domain of $V$ (with respect to $\kappa$)} is the open set $\fdom V\subset X$ defined by
\[
\fdom V = \kappa^{-1}\bigl(\ft V\bigr)\;.
\]
\item 
The \emph{projection family of $V$} is the mapping
\[\pf V: A\times \fdom V\to V\]
defined by the condition \[V\cap(a+\kappa(x))=\{\pf V(a,x)\}\;.\]
(The map $\pf V$ defines a family of projections $a\mapsto\pf V(a,x)$ of $A$ onto $V$ indexed by $x\in \fdom V$, hence the terminology --- \cf\ Lemma~\ref{singleton}.)

\item
The \emph{fiber family of $V$} is the mapping
\[
\ff V: E_{\fdom V}\to V
\]
defined by the condition
\[\ff V\circ q = \pf V\;.\]
(This is well defined because $\pf V(a,x)=\pf V(b,x)$ if $a-b\in \kappa(x)$, and it defines a family of isomorphisms $E_x\cong V$ indexed by $x\in\fdom V$.)
\end{enumerate}

\begin{lemma}\label{usedonlyonce}
Let $A$ be a Hausdorff vector space, let $(\pi:E\to X,A,q)$ be a rank-$n$ quotient vector bundle whose kernel map $\kappa:X\to\Max_n A$ is Fell-continuous, and let $V\subset A$ be an $n$-dimensional linear subspace. For all subsets $W\subset V$, all $a\in A$, and all $x\in \fdom V$, the following conditions are equivalent:
\begin{enumerate}
\item $\ff V(q(a,x))\in W$ (equiv., $\pf V(a,x)\in W$);
\item $\bigl(a+\kappa(x)\bigr)\cap(V\setminus W)=\emptyset$.
\end{enumerate}
\end{lemma}

\begin{proof}
Immediate consequence of the definitions of $\ff V$ and $\pf V$.
\qedhere
\end{proof}

\begin{lemma}\label{lemmaJP-VxUV->E-3}
Let $A$ be a Hausdorff vector space, and let
\[(\pi:E\to \Max_n A,A,q)\]
be the rank-$n$ quotient vector bundle whose kernel map is the identity on $\Max_n A$.
Let also $V\subset A$ be an $n$-dimensional linear subspace, let $W\subset V$ be an open subset (in the subspace topology of $V$),
and let $a\notin V$.
Fix some norm on $V\oplus\linspan{a}$, let
$r_{V\oplus\linspan{a}}: \bigl(V\oplus\linspan{a}\bigr)\setminus\{0\}\to S_{V\oplus\linspan{a}}$ be the retraction $r_{V\oplus\linspan{a}}(v)=v/\|v\|$ onto the unit sphere 
$S_{V\oplus\linspan{a}}$ and
let \[K_W=\overline{r_{V\oplus\linspan{a}}\bigl((V\setminus W)-a\bigr)}\;.\] Then $K_W$ is compact. Moreover, for all $P\in\ft V$ we have 
\begin{equation}\label{Wequiv}
\pf V(a,P)\in W\iff P\in\ft K_W\;.
\end{equation}
\end{lemma}

\begin{remark*}
Since $S_{V\oplus\linspan{a}}$ is closed in $A$, the closure of $r_{V\oplus\linspan{a}}\bigl((V\setminus W)-a\bigr)$ is the same in $A$
or in $S_{V\oplus\linspan{a}}$.
\end{remark*}

\begin{proof}
In order to show that $K_W$ is compact notice that 
$S_{V\oplus\linspan{a}}$ 
is compact because $V\oplus\linspan{a}$ is finite dimensional. 
Since $K_W\subset S_{V\oplus\linspan{a}}$ is closed, $K_W$ is compact. Now let $P\in\ft V$.
The condition $P\in \ft K_W$ is, by definition, equivalent to
\begin{equation}\label{eqrW0}
P\cap K_W=\emptyset\;,
\end{equation}
which implies
\begin{equation}\label{eqrW1}
P\cap r_{V\oplus\linspan{a}}\bigl((V\setminus W)-a\bigr)=\emptyset\;,
\end{equation}
and, equivalently,
\begin{equation}\label{eqrW2}
P\cap \bigl((V\setminus W)-a\bigr)=\emptyset\;,
\end{equation}
which in turn is equivalent, by Lemma~\ref{usedonlyonce}, to
\begin{equation}\label{eqrW3}
\pf V(a,P)\in W\;.
\end{equation}
So we have proved the $\Leftarrow$ implication of \eqref{Wequiv}, and also concluded that the conditions \eqref{eqrW1}--\eqref{eqrW3} are all equivalent. Hence, in order to prove the $\Rightarrow$ implication of \eqref{Wequiv} we only need to show that \eqref{eqrW1} implies \eqref{eqrW0}. Suppose the former holds, and let $w\in K_W$. Then, since $V\oplus\linspan{a}$ 
is normed, there is a sequence $(v_k)$ in $(V\setminus W)-a$ such that $\lim v_k/\|v_k\|=w$. If
$(v_k)$ is not bounded we also have $\lim(v_k+a)/\|v_k+a\|=w$. In this case, since the sequence $(v_k+a)$ has values in $V$, we obtain $w\in V$, which implies that $w\notin P$ because $P\in \ft V$, so \eqref{eqrW0} holds. Now assume that $(v_k)$ is bounded. Then, since $V$ is finite dimensional and $V\setminus W-a$ is closed, there is a convergent subsequence $(w_k)$ such that $\lim w_k\in V\setminus W-a$, and thus $w=r_{V\oplus\linspan{a}}(\lim w_k)\in r_{V\oplus\linspan{a}}\bigl((V\setminus W)-a\bigr)$. Therefore $w\notin P$ and again we conclude that \eqref{eqrW0} holds.
 \qedhere
\end{proof}

\begin{theorem}
If $A$ is a Hausdorff vector space then $\Max_n A$, with the Fell topology, is a Hausdorff space.
\end{theorem}

\begin{proof}
Let $P,Q\in\Max_n A$ be such that $P\neq Q$. By 
Lemma~\ref{lemmaJP-linearalgebra}
there is an $n$-dimensional subspace $V\subset A$ such that
$V\cap P=V\cap Q=\{0\}$.
Let $(\pi:E\to \Max_n A,A,q)$ be the rank-$n$ quotient vector bundle whose kernel map is the identity on $\Max_n A$.
Let $a\in P\setminus Q$. Then $\hat a(P)=0$ and $\hat a(Q)\neq 0$.
Let $s: \fdom V\to V$ be defined by $s=\ff V\circ \hat a$.
Then $s(P)=0$  and $s(Q)\neq 0$, and, since $V$ is Hausdorff, we only need to show that $s$ is continuous.
But this follows from Lemma~\ref{lemmaJP-VxUV->E-3} because for all $R\in\fdom V$ we have $s(R)=\pf V\bigl(a,R\bigr)$, and thus for each open set $W$ of $V$ there is a compact set $K_W\subset A$ such that the preimage $s^{-1}(W)$ equals $\ft K_W$.  \qedhere
\end{proof}

\paragraph{Local triviality.} Let $A$ be a Hausdorff vector space, and let \[(\pi:E\to X,A,q)\] be a rank-$n$ quotient vector bundle with Fell-continuous kernel map \[\kappa:X\to\Max_n A\;.\] Let also $V\subset A$ be an $n$-dimensional linear subspace. The restriction of $q$ to $A\times\fdom V$ is an open map because the fiber domain $\fdom V$ is open, and therefore the projection family $\pf V$ is continuous if and only if the fiber family $\ff V$ is. The continuity of these maps is closely related to the existence of local trivializations, as we now see:

\begin{lemma}\label{lemmaJP-VxUV->E}
With the bundle $(\pi,A,q)$ and $V$ as just defined, let
\[q': V\times \fdom V\to E_{\fdom V}\]
be the restriction of $q$ to $V\times \fdom V$. Then $q'$ is continuous and bijective, and the following conditions are equivalent:
\begin{enumerate}
\item $q'$ is an isomorphism of bundles in $\LBun(X)$;
\item $\ff V$ is continuous (equiv., $\pf V$ is continuous).
\end{enumerate}
\end{lemma}

\begin{proof}
The function $q'$ is clearly continuous and it is bijective by Lemma~\ref{lemmaJP-linindep}. Its inverse is the pairing $\langle \ff V,\pi_{\fdom V}\rangle$, and thus $q'$ is a homeomorphism if and only if $\ff V$ is continuous.
\qedhere
\end{proof}

\begin{theorem}\label{fVcontsuffices}
Let $A$ be a Hausdorff vector space and let \[(\pi:E\to X,A,q)\] be a rank-$n$ quotient vector bundle with Fell-continuous kernel map \[\kappa:X\to\Max_n A\;.\]
The following conditions are equivalent:
\begin{enumerate}
\item\label{fV1} $\pi$ is locally trivial;
\item\label{fV2} For all $n$-dimensional subspaces $V\subset A$ the fiber family $\ff V$ (equiv., the projection family $\pf V$) is continuous.
\end{enumerate}
\end{theorem}

\begin{proof}
The implication $\eqref{fV2}\Rightarrow \eqref{fV1}$ is a consequence of Lemma~\ref{lemmaJP-VxUV->E}: for each $x\in X$ choose an $n$-dimensional subspace $V\subset A$ such that $V\cap\kappa(x)=\{0\}$; the fiber domain $\fdom V$ is an open neighborhood of $x$ and the continuity of $\ff V$ implies that there is a local trivialization $(q')^{-1}:E_{\fdom V}\to V\times\fdom V$, where $q'$ is the restriction of $q$ to $V\times\fdom V$.

Now let us prove the implication $\eqref{fV1}\Rightarrow \eqref{fV2}$. Suppose $\pi$ is locally trivial and let $V\subset A$ be an $n$-dimensional subspace. Let $x\in \fdom V$, let $q':V\times\fdom V\to E_{\fdom V}$ be the restriction of $q$ to $V\times\fdom V$, and let $\phi: E_U\to V\times U$
be a trivialization on some open neighborhood $U$ of $x$. We may assume that $U\subset \fdom V$, and thus obtain a continuous and bijective function $\phi\circ q'': V\times U\to V\times U$,
which is a morphism in $\LBun(U)$, where $q''$ is the restriction of $q'$ to $V\times U$.
The function $\phi\circ q''$ is necessarily a bundle isomorphism because it restricts fiberwise to continuous linear isomorphisms $\phi_x:V\times \{x\}\to V\times \{x\}$, which are necessarily homeomorphisms because $V$ has the Euclidean topology, and thus $q''$ is a homeomorphism. Since $q''$ is a generic restriction of $q'$ and the domains of all such maps $q''$ cover the domain of $q'$, we conclude that $q'$ is itself a homeomorphism. So, by Lemma~\ref{lemmaJP-VxUV->E}, $\ff V$ is continuous.  \qedhere
\end{proof}

\paragraph{Two technical lemmas.}

In the following two lemmas $A$ is an arbitrary Hausdorff vector space, $(\pi:E\to X,A,q)$ is a rank-$n$ quotient vector bundle with Fell-continuous kernel map $\kappa:X\to\Max_n A$, and $V\subset A$ is an $n$-dimensional subspace. 

\begin{lemma}\label{lemmaJP:fvcontiffxinint}
Let $\mathcal B$ be a basis of neighborhoods of the origin in the topology of $V$. Then
the map $\ff V: E_{\fdom V}\to V$ is continuous if and only if
for any $W\in\mathcal B$ and for any $x\in \fdom V$
there is a neighborhood $U\subset A$ of zero in the topology of $A$ 
such that $x$ is in the open set
\[
N_{W,U}=\interior\bigl\{y\in \fdom V\st \kappa(y)\cap\bigl((V\setminus W)+U\bigr)=\emptyset\bigr\}\;.
\]
\end{lemma}

\begin{proof}
First, assuming the hypothesis about the sets $N_{W,U}$, we show that $\ff V$ is continuous.
Let $[a,x]=q(a,x)\in E_{\fdom V}$ and let $v=\ff V([a,x])$. The collection
$\{v+W\st W\in\mathcal B\}$ is a basis of neighborhoods of $v$.
In order to prove continuity of $\ff V$ we shall, given an arbitrary but fixed $W\in\mathcal B$, find an open set $\mathcal U\subset E_{\fdom V}$ containing $[a,x]$ such that
\begin{equation}\label{fvcontinuity}
\ff V(\mathcal U)\subset v+W\;.
\end{equation}
By hypothesis there is a neighborhood
of zero $U\subset A$ such that $x\in N_{W,U}$, and we define $\mathcal U=q((v-U)\times N_{W,U})$. So let us prove \eqref{fvcontinuity}. Let
$y\in N_{W,U}$ and $b\in v-U$. Then 
$\bigl(\kappa(y)-U\bigr)\cap(V\setminus W)=\emptyset$ and, since
$b-v\in -U$, we find that 
$\bigl(\kappa(y)+b-v\bigr)\cap(V\setminus W)=\emptyset$,
so $\bigl(\kappa(y)+b-v\bigr)\cap V\subset W$, and thus 
$\bigl(\kappa(y)+b\bigr)\cap V\subset v+W$. We have shown that
$\ff V([y,b])\in v+W$ for all $[y,b]\in\mathcal U$, thus proving \eqref{fvcontinuity}.

Reciprocally, assume $\ff V$ is continuous. Then, by the continuity
of $\pf V=q\circ \ff V$,
for any $W\in\mathcal B$ and for any $x\in \fdom V$
there is an open neighborhood $U$ of zero in $A$ and an open neighborhood $N$ of $x$
such that $\ff V(q((-U)\times N))\subset W$. Hence, for any $y\in N$ and $b\in U$
we have $\bigl(\kappa(y)-b\bigr)\cap (V\setminus W)=\emptyset$.
But this implies that $\bigl(\kappa(y)-U\bigr)\cap (V\setminus W)=\emptyset$
for any $y\in N$, so we conclude that $N\subset N_{W,U}$.
\qedhere
\end{proof}

\begin{lemma}\label{lemmaJP:finerthanfellbutnotmuch}
The following conditions are equivalent:
\begin{enumerate}
\item\label{finer1} For any $b\in A$ and any $x\in X$ such that $\hat b(x)\neq 0_x$ there
is a neighborhood $U$ of zero in $A$ such that
\[
x\in\interior\bigl\{y\in X\st \kappa(y)\cap(b+U)=\emptyset\bigr\}\;;
\]
\item\label{finer2} For any compact set $K\subset A$ and any $x\in \kappa^{-1}(\check K)$ there
is a neighborhood $U$ of zero in $A$ such that
\[
x\in\interior\bigl\{y\in X\st\kappa(y)\cap(K+U)=\emptyset\bigr\}\;.
\]
\end{enumerate}
\end{lemma}

\begin{proof}
Trivially $\eqref{finer2}\Rightarrow\eqref{finer1}$, so assume \eqref{finer1} holds. Then for any $b\in K$ 
there is a neighborhood $U_b\subset A$ of zero such that
\[
x\in\interior\bigl\{y\in X\st\kappa(y)\cap(b+U_b)=\emptyset\bigr\}\;.
\]
Pick a neighborhood $W_b$ of zero in $A$ such that $W_b+W_b\subset U_b$.
Then \[K\subset \bigcup_b(b+W_b)\;,\] so we can pick a finite covering
\[K\subset \bigcup_{i=1}^n(b_i+W_{b_i})\;.\]
Let $U=\bigcap_i W_{b_i}$.
Then $K+U\subset \bigcup_i(b_i+W_{b_i}+U)\subset \bigcup_i(b_i+U_{b_i})$, and therefore
\begin{eqnarray*}
x&\in& \bigcap_i\interior\bigl\{y\in X\st\kappa(y)\cap(b_i+U_{b_i})=\emptyset\bigr\}\\
&=&\interior\bigcap_i\bigl\{y\in X\st\kappa(y)\cap(b_i+U_{b_i})=\emptyset\bigr\}\\
&=&\interior \left\{y\in X\st\kappa(y)\cap\Bigl(\bigcup_i(b_i+U_{b_i})\Bigr)=\emptyset\right\}\\
&\subset&\interior \bigl\{y\in X\st\kappa(y)\cap(K+U)=\emptyset\bigr\}\;.
 \qedhere
\end{eqnarray*}
\end{proof}

\paragraph{Locally convex spaces.}

Using the previous results we can provide another equivalent condition to local triviality, in the case where the Hausdorff vector space $A$ is locally convex:

\begin{theorem}\label{thmJP-loctriviff}
Let $A$ be a Hausdorff locally convex space, and let $(\pi:E\to X,A,q)$ be a rank-$n$ quotient vector bundle with Fell-continuous kernel map $\kappa:X\to\Max_n A$. 
Then $\pi$ is locally trivial if and only if
for any $(b,x)\in A\times X$ such that $\hat b(x)\neq 0_x$
there is a neighborhood $U$ of zero in $A$ such that 
\[
x\in\interior\bigl\{y\in X\st\kappa(y)\cap(b+U)=\emptyset\bigr\}\;.
\]
\end{theorem}
[We remark that only the reverse implication uses local convexity.]

\begin{proof}
Assume first that $\pi$ is locally trivial. Then, by Theorem~\ref{fVcontsuffices}, $\ff V$ is continuous for any linear subspace $V\subset A$
of dimension $n$. Let $(b,x)\in A\times X$ be such that $\hat b(x)\neq 0_x$; that is, $b\notin \kappa(x)$.
Then we can choose an $n$-dimensional vector subspace $V\subset A$ such that $b\in V$ and $V\cap\kappa(x)=\{0\}$ (let $V$ be spanned by $b,a_1,\ldots, a_{n-1}$ such that $b+\kappa(x),a_1+\kappa(x),\ldots,a_{n-1}+\kappa(x)$ is a basis of $A/\kappa(x)$).
Let $W\subset V$ be 
a neighborhood of zero in $V$ with $b\notin W$. By Lemma~\ref{lemmaJP:fvcontiffxinint}, there is a neighborhood $U$ of zero in $A$
such that
\[
x\in\interior\bigl\{y\in \fdom V\st \kappa(y)\cap\bigl((V\setminus W)+U\bigr)=\emptyset\bigr\}\;.
\]
Then, because $b\in V\setminus W$, we obtain
\begin{eqnarray*}
x&\in&\interior\bigl\{y\in X\st \kappa(y)\cap\bigl((V\setminus W)+U\bigr)
=\emptyset\bigr\}\\
&\subset& \interior\bigl\{y\in X\st \kappa(y)\cap\bigl(b+U\bigr)=\emptyset\bigr\}\;.
\end{eqnarray*}

In order to prove the converse let us assume that for any $(b,x)\in A\times X$ such that $\hat b(x)\neq 0_x$
there is a neighborhood $U$ of zero in $A$ satisfying 
\[
x\in\interior\bigl\{y\in X\st\kappa(y)\cap(b+U)=\emptyset\bigr\}\;.
\]
Equivalently, by Lemma~\ref{lemmaJP:finerthanfellbutnotmuch}, we assume that for any compact set $K\subset A$ and any $x\in\kappa^{-1}(\check K)$ there is a neighborhood $U\subset A$ of zero such that
\begin{equation}\label{eqK+U}
x\in\interior\bigl\{y\in X\st\kappa(y)\cap(K+U)=\emptyset\bigr\}\;.
\end{equation}
Since $A$ is locally convex we shall assume that $U$ is convex. Let $V\subset A$ be an arbitrary $n$-dimensional linear subspace, and
choose some norm on $V$. Let also $\mathcal B$ be the collection of balls around the origin of $V$ in that norm, let $W\in \mathcal B$, and let $x\in\fdom V$.
Now let $K$ be the (compact) boundary of $W$ in $V$. Then $x\in\kappa^{-1}(\check K)$, and \eqref{eqK+U} holds. 
We claim that for any linear subspace $P\subset A$ we have
$P\cap \bigl((V\setminus W)+U\bigr)=\emptyset$ if and only if $P\cap(K+U)=\emptyset$. 
One implication is clear because $K\subset V\setminus W$, so let $b\in P\cap \bigl((V\setminus W)+U\bigr)$.
Then $b=v+u$ with $v\in V\setminus W$ and $u\in U$, so, since $W$ is a ball, there is $c\in (0,1)$ such that $cv\in K$. Then $cu\in U$
because $U$ is convex and $0\in U$, and thus $cb\in P\cap (K+U)$, which proves the claim.
Hence, \eqref{eqK+U} is equivalent to
\[
x\in\interior\bigl\{y\in X\st \kappa(y)\cap\bigl((V\setminus W)+U\bigr)=\emptyset\bigr\}\;,
\]
which in turn is equivalent to the condition $x\in N_{W,U}$ of Lemma~\ref{lemmaJP:fvcontiffxinint} because we are assuming that $x$ belongs to the open set $\fdom V$. Hence, by Lemma~\ref{lemmaJP:fvcontiffxinint}, we have proved that $\ff V$ is continuous for all $n$-dimensional subspaces $V\subset A$, and thus, by Theorem~\ref{fVcontsuffices}, $\pi$ is locally trivial.
\qedhere
\end{proof}

\paragraph{Normed bundles and Banach bundles.} Let us conclude our study of local triviality by looking at continuous normed bundles and Banach bundles.

\begin{theorem}\label{normedloctriv}
Let $(\pi:E\to X,A,q)$ be a normed rank-$n$ quotient vector bundle with kernel map $\kappa:X\to\Max_n A$. If $\kappa$ is continuous with respect to the closed balls topology then:
\begin{enumerate}
\item\label{nl1} For every $n$-dimensional subspace $V\subset A$,
the bundle $E$ induced by $\kappa$ is trivial on the open set
$\fdom V=\kappa^{-1}\bigl(\ft V\bigr)$;
\item\label{nl2} $\pi$ is locally trivial.
\end{enumerate}
\end{theorem}

\begin{proof}
The space $A$ is Hausdorff and, by Lemma~\ref{cbcontainsfell}, the map $\kappa$ is Fell-continuous. Hence, by Lemma~\ref{lemmaJP-VxUV->E} and Theorem~\ref{fVcontsuffices}, the two conditions \eqref{nl1} and \eqref{nl2} are equivalent. So let us apply Theorem~\ref{thmJP-loctriviff} (since $A$ is also locally convex). Given $(b,x)\in A\times X$ with $\hat b(x)\neq 0_x$ (equivalently, $b\notin\kappa(x)$) let us define
\begin{eqnarray*}
\varepsilon&=&d\bigl(b,\kappa(x)\bigr)/2\;,\\
U&=&B_\varepsilon(0)\;.
\end{eqnarray*}
Then
\[
x\in \kappa^{-1}(\cb b\varepsilon)\subset \bigl\{y\in X\st\kappa(y)\cap(b+U)=\emptyset\bigr\}\;,
\]
and it follows from Theorem~\ref{thmJP-loctriviff} that $\pi$ is locally trivial.
\qedhere
\end{proof}

This implies the following result, which is mentioned but not proved in \cite{FD1}*{p.\ 129}:

\begin{corollary}
Every Banach bundle of constant finite rank on a locally compact Hausdorff space is locally trivial.
\end{corollary}

\begin{proof}
Recall from Theorem~\ref{banachasquotient} that every Banach bundle $\pi:E\to X$ with $X$ locally compact Hausdorff can be made a quotient vector bundle $(\pi,A,q)$ by taking $A=\sections_0(\pi)$ with the supremum norm and $q=\eval$. By Corollary~\ref{classbanach} this bundle is classified by a kernel map $\kappa$ which is continuous with respect to the closed balls topology. The conclusion follows from Theorem~\ref{normedloctriv}. \qedhere
\end{proof}

\section{Grassmannians}

The role played by the spaces $\Sub A$ in the classification of quotient vector bundles $(\pi,A,q)$ is analogous to that of Grassmannians in the classification of locally trivial finite rank vector bundles. However, the relations between $\Sub A$ and Grassmannians in the preceding sections are obscured by the fact that $A$ has in general been infinite dimensional, so our purpose in the present section is to investigate the extent to which the topologies of $\Sub A$ and Grassmannians coincide. This said, 
this section is largely independent from the preceding ones. We shall only need to consider the lower Vietoris topology on $\Sub A$, and the main results of this section will require $A$ to be Hausdorff.

\paragraph{Linearly independent open sets.}

Let $A$ be a topological vector space.
We say that a finite collection of open sets $U_1,\dots,U_k\subset A$
is \emph{linearly independent} if any $k$ vectors $v_1,\dots,v_k$ with $v_i\in U_i$
are linearly independent.

\begin{lemma}\label{lemmaJP:linindepopensets}
Let $A$ be a Hausdorff vector space, let $S\subset A$ be a finite set, and for each $v\in S$ let $U_v$ be a neighborhood of $v$.
Let $e_1,\dots,e_m$ be a basis of the linear span $\linspan{S}$. 
Then there are linearly independent neighborhoods $U_1',\dots,U_m'$ of
$e_1,\dots,e_m$ such that, for any $v\in S$, if we write $v=\sum_ja_{j,v}e_j$ (with $a_{j,v}\in\CC$)
then $\sum_ja_{j,v}U_j'\subset U_v$.
\end{lemma}
\begin{proof}
We divide the proof into three steps:
\begin{enumerate}
\item\label{cond1} We claim that there are linearly independent
neighborhoods $U_j^{\text{l.i.}}$ of $e_j$ (with $j=1,\dots,m$).
Consider the unit sphere $S^{2m-1}\subset\mathbb C^m$ and let
$f: S^{2m-1}\times A^m\to A$ be the function
$f(z_1,\dots,z_m,u_1,\dots,u_m)=\sum_jz_ju_j$. Then
$S^{2m-1}\times\{(e_1,\dots,e_m)\}\subset f^{-1}(A\setminus 0)$ so,
since $S^{2m-1}$ is compact, there are (by the tube lemma) open neighborhoods $U_j^{\text{l.i.}}$ of
$e_j$ such that $S^{2m-1}\times\prod U_j^{\text{l.i.}}\subset f^{-1}(A\setminus 0)$. This proves the claim.
\item\label{cond2} Now let $v=\sum a_{j,v}e_j\in S$.
We claim that there are neighborhoods $U_{j,v}$
of $e_j$ such that $\sum_j a_{j,v} U_{j,v}\subset U_v$.
Consider the function $f_v: A^m\to A$ given by
$f_v(u_1,\dots,u_m)=\sum_ja_{j,v}u_j$. Since $f_v$ is continuous, there are neighborhoods
$U_{j,v}$ of $e_j$ such that $f_v(U_{1,v}\times\dots\times U_{m,v})\subset U_v$, which proves the claim.
\item Now, for each $j=1,\ldots,m$, let
\[U_j'=U_j^{\text{l.i.}}\cap \left(\bigcap_{v\in S}U_{j,v}\right)\;.\] By \eqref{cond1}, the sets
$U_j'$  (which are open because $S$ is finite) are linearly independent. And, by \eqref{cond2}, we have $\sum_j a_{j,v}U'_j\subset U_v$,
which concludes the proof. \qedhere
\end{enumerate}
\end{proof}

\paragraph{Grassmannians as subspaces.}

For any Hausdorff vector space $A$ (of any dimension) and any integer $k>0$, we shall write $V(k,A)$ for the set of injective linear maps from $\CC^k$ to $A$ with the product topology, and we shall write $S(v,U)$, where $v\in\CC^k$ and $U\subset A$ is open, for the sub basic open set consisting of those $\phi\in V(k,A)$ such that $\phi(v)\in U$. Denoting by $\Gr(k,A)$ the set of all the $k$-dimensional subspaces of $A$, we have a surjective map
\[
p_k:V(k,A)\to\Gr(k,A)\;,
\]
defined by
$p_k(\phi) = \image\phi$.
We shall refer to $\Gr(k,A)$, equipped with the quotient topology, as a  \emph{Grassmannian of $A$} (so we have a homeomorphism $\Gr(k,A)\cong V(k,A)/\GL(k,\CC)$).
This, of course, is a generalization of the usual definition of Grassmannian of a finite dimensional vector space.

\begin{theorem}\label{grsubspace}
Let $A$ be a Hausdorff vector space. For any integer $k>0$, the Grassmannian $\Gr(k,A)$ is a topological subspace of $\Sub A$.
\end{theorem}

\begin{proof}
Let us first prove that, given an open set $U\subset A$, the set
\[
\ps U\cap\Gr(k,A)
\]
is open in the quotient topology.
We only need to show that the set
\[p_k^{-1}\bigl(\ps U\cap\Gr(k,A)\bigr)=\bigl\{\phi\in V(k,A)\st \image\phi\cap U\neq\emptyset\}\;,\]
is open in $V(k,A)$.
Given $\phi\in p_k^{-1}\bigl(\ps U\cap\Gr(k,A)\bigr)$, there is $v\in\CC^k$ such that $\phi(v)\in U$. Then clearly
$\phi\in S(v,U)\subset p_k^{-1}\bigl(\ps U\cap\Gr(k,A)\bigr)$, which completes the proof.

Now we prove that any set $W\subset \Gr(k,A)$ which is open in the quotient topology is also open 
in the topology induced by $\Sub A$. Let $P$ be an arbitrary element of $W$, with $W$ open in the quotient topology. In order to show that $W$ is also open in the subspace topology we shall find an open set $\mathcal U\subset \Sub A$ such that
\begin{equation}\label{eq:grassSub}
P\in\Gr(k,A)\cap\mathcal U \subset W\;.
\end{equation}
First we observe that $P=p_k(\phi)$ for some $\phi\in p_k^{-1}(W)$, so there are
\[u_1,\dots,u_n\in\CC^k\] and open sets \[U_1,\dots,U_n\subset A\]
such that
\begin{equation}\label{eq:SuiUiW}
\phi\in\bigcap_{i=1}^n S(u_i,U_i)\subset p_k^{-1}(W)\;.
\end{equation}
We may assume without loss of generality that for some $m\leq n$ the vectors $u_1,\dots,u_m$ form a basis of $\linspan{u_1,\dots,u_n}$.
Then we are within the conditions of Lemma~\ref{lemmaJP:linindepopensets} with
\[
\begin{array}{rcl}
S&=&\{\phi(u_1),\ldots,\phi(u_n)\}\;,\\
\phi(u_i)&\in& U_i\ \ \ \ \ \ (i=1,\ldots,n)\;,\\
e_j&=&\phi(u_j)\ \ \ \ (j=1,\ldots,m)\;.
\end{array}
\]
Let $U_j'$ be as in Lemma~\ref{lemmaJP:linindepopensets}.
Then $P\in\bigcap_{j=1}^m\ps U'_j$, so let us define
\[
\mathcal U = \bigcap_{j=1}^m\ps U'_j\;,
\]
and let us show that the inclusion in \eqref{eq:grassSub} holds:
\begin{equation}\label{eq:grassSub2}
\Gr(k,A)\cap\mathcal U \subset W
\end{equation}
Let $Q\in\Gr(k,A)\cap\mathcal U$. Then $Q=p_k(\psi)$ for some $\psi\in V(k,A)$. For each $j=1,\dots,m$ we have $Q\in\ps U_j'$, so there is $w_j\in\CC^k$ such that
$\psi(w_j)\in U'_j$. The open sets $U'_j$ are linearly independent, so the vectors $w_j$ are linearly independent, too. Hence, there is
$g\in \GL(k,\CC)$ such that $g(u_j)=w_j$ for all $j=1,\dots,m$, and
\[
Q=\image(\psi\circ g)=p_k(\psi\circ g)\;.
\]
Now, for each $i=1,\dots,n$ we can write
$u_i=\sum_{j=1}^ma_{ij}u_j$, with $a_{ij}\in\CC$. Then
\[\psi(g(u_i))=\sum_{j=1}^m a_{ij}\psi(g(u_j))=\sum_{j=1}^m a_{ij}\psi(w_j)\in\sum_{j=1}^m a_{ij}U_j'\subset U_i\;,\]
and thus $\psi \circ g\in S(u_i,U_i)$. Since, by \eqref{eq:SuiUiW}, we have
\[
\bigcap_{i=1}^n S(u_i,U_i)\subset p_k^{-1}(W)\;,
\]
we obtain $Q=p_k(\psi \circ g)\in W$, showing that \eqref{eq:grassSub2} holds. \qedhere
\end{proof}

\begin{corollary}
For any $k\in\{1,\ldots,n\}$ we have a homeomorphism
\[\Max_k \CC^n\cong \Gr(n-k,\CC^n)\;.\]
\end{corollary}

\paragraph{A new basis for the lower Vietoris topology.}

Let us provide a description of the lower Vietoris topology in terms of the Grassmannian topologies. In order to do this, first we need a finer subbasis for the lower Vietoris topology:

\begin{lemma}\label{lemmaJP:basislinindep}
Let $A$ be a Hausdorff vector space.
Then the topology of $\Sub A$ has as a basis the collection of
finite intersections $\bigcap_{i=1}^k\ps U_i$ where
$U_1,\dots,U_k\subset A$ are linearly independent open sets.
\end{lemma}
\begin{proof}
Given open sets $U_1,\dots,U_n\subset A$ let 
$P\in\bigcap_i\ps U_i$. We want to find 
linearly independent open sets $U_1',\dots,U_m'$ such that
$P\in\bigcap_{j=1}^m \ps U_j'\subset\bigcap_{i=1}^n\ps U_i$.
For each $i$ let $e_i\in P\cap U_i$. We may assume without loss of generality
that there is some $m\leq n$ such that
$e_1,\dots,e_m$ form a basis of 
$\linspan{e_1,\dots,e_n}$. Then we are within the conditions of Lemma~\ref{lemmaJP:linindepopensets} with $S=\{e_1,\dots,e_n\}$, so let
$U_1',\dots,U_m'$ be as in Lemma~\ref{lemmaJP:linindepopensets}. Then we have
\[P\in\bigcap_{j=1}^m \ps U'_j\;,
\]
and we need to show that
\begin{equation}\label{lasteq}
\bigcap_{j=1}^m \ps U_j'\subset\bigcap_{i=1}^n\ps U_i\;.
\end{equation}
Let $Q\in\bigcap_{j=1}^m \ps U_j'$.
For each $j=1,\dots,m$ let $u_j\in Q\cap U'_j$. By Lemma~\ref{lemmaJP:linindepopensets} there are $a_{ij}\in\CC$ such that,
for each $i=1,\dots,n$ we have
\[\sum_{j=1}^m a_{ij}u_j\in\sum_{j=1}^m a_{ij}U_j'\subset U_i\;.\]
Since we also have $\sum_{j=1}^m a_{ij}u_j\in Q$,
it follows that $Q\in \ps U_i$ for all $i$. Therefore we have $Q\in\bigcap_{i=1}^n\ps U_i$, which proves \eqref{lasteq}. \qedhere
\end{proof}

Now we shall use the notation $\upsegment W$, for $W$ any subset of $\Sub A$, to denote the \emph{upper-closure} of $W$ in the inclusion order of $\Sub A$:
\[
\upsegment W=\{V\in\Sub A\st \exists_{V'\in W}\ V'\subset V\}\;.
\]

\begin{theorem}
Let $A$ be a Hausdorff vector space. A basis for the lower Vietoris topology of $\Sub A$ consists of
the collection of all the sets $\upsegment W$ where for some $k\in\NN_{>0}$ the set $W\subset \Gr(k,A)$ 
is open in the Grassmannian topology.
\end{theorem}

\begin{proof}
Due to Theorem~\ref{grsubspace} and Lemma~\ref{lemmaJP:basislinindep} it is enough to show that
if $U_1,\dots,U_k\subset A$ are linearly independent open sets, and if we let \[W=\Gr(k,A)\cap \bigcap_{i=1}^k\ps U_i\;,\] then
$\upsegment W=\bigcap\ps U_i$. The inclusion $\upsegment W\subset \bigcap\ps U_i$ is immediate because $W\subset\bigcap\ps U_i$ and the open sets of $\Sub A$ are upwards closed: $\upsegment W\subset\upsegment\bigl({\bigcap\ps U_i}\bigr)=\bigcap\ps U_i$. For the other inclusion, consider
$P\in \bigcap\ps U_i$. Then there are linearly independent vectors $v_1,\ldots,v_k\in A$ such that $v_i\in P\cap U_i$ for all $i$. Let $V=\linspan{v_1,\dots,v_k}$. Then $V\in W$ and $V\subset P$, and thus $P\in\upsegment W$. \qedhere
\end{proof}

\appendix

\section{Appendix}

\begin{lemma}\label{singleton}
Let $A$ be a vector space and $n$ a positive integer. Let also $V,P\subset A$ be linear subspaces such that $V\cap P=\{0\}$ and $\dim(V)=\dim(A/P)=n$. Then $A=V\oplus P$, and the set $V\cap(a+P)$ is a singleton.
\end{lemma}

\begin{proof}
Let $a_1,\ldots, a_n$ be a basis of $V$. Then $a_1+P,\ldots,a_n+P$ is a basis of $A/P$, and there is an isomorphism $\iota:A/P\to V$ defined by $a_i+P\mapsto a_i$, which gives us a projection $\pr P:A\to A$ by composing the quotient $a\mapsto a+P$ with $\iota$ and the inclusion $V\to A$:
\[
A\to A/P\stackrel\iota\to V\to A\;.
\]
Hence, we get $\pr P(A)=V$ and $\ker \pr P=P$ (and thus $A=V\oplus P$), and therefore $V\cap(a+P)=\{\pr P(a)\}$. \qedhere
\end{proof}

\begin{lemma}\label{lemmaJP-linearalgebra}
Let $A$ be a vector space and $n$ a positive integer.
Let $P,Q\subset A$ be linear subspaces such that $\dim(A/P)=\dim(A/Q)=n$. Then there is an
$n$-dimensional subspace $V\subset A$ such that $V\cap P=V\cap Q=\{0\}$.
\end{lemma}

\begin{proof}
We begin by fixing an isomorphism
\begin{multline*}
\phi: A\to (P\cap Q)\oplus A/(P+Q)\oplus (P+Q)/(P\cap Q)\\
\cong (P\cap Q)\oplus A/(P+Q)\oplus P/(P\cap Q)\oplus Q/(P\cap Q)\;.
\end{multline*}
Notice that $\phi(P)\subset (P\cap Q)\oplus P/(P\cap Q)$ and $\phi(Q)\subset(P\cap Q)\oplus Q/(P\cap Q)$.
Let $j=\dim A/(P+Q)$. Then, since $A/P\cong A/(P+Q)\oplus(P+Q)/P\cong A/(P+Q)\oplus Q/(P\cap Q)$,
we have $\dim Q/(P\cap Q)=n-j$ and, similarly, $\dim P/(P\cap Q)=n-j$.
Now pick an $(n-j)$-dimensional subspace
$W\subset P/(P\cap Q)\oplus Q/(P\cap Q)$
such that $W\cap P/(P\cap Q)=W\cap Q/(P\cap Q)=\{0\}$, and let $V=\phi^{-1}\bigl(W\oplus A/(P+Q)\bigr)$.
Then $V\cap P=V\cap Q=\{0\}$.  \qedhere
\end{proof}

\vspace*{5mm}
\noindent {\sc
Centro de An\'alise Matem\'atica, Geometria e Sistemas Din\^amicos
Departamento de Matem\'{a}tica, Instituto Superior T\'{e}cnico\\
Universidade de Lisboa\\
Av.\ Rovisco Pais 1, 1049-001 Lisboa, Portugal}\\
{\it E-mail:} {\sf pmr@math.tecnico.ulisboa.pt}, {\sf jsantos@math.tecnico.ulisboa.pt}


\begin{bibdiv}

\begin{biblist}

\bib{FeldmanMooreI-II}{article}{
  author={Feldman, Jacob},
  author={Moore, Calvin C.},
  title={Ergodic equivalence relations, cohomology, and von Neumann algebras. I, II},
  journal={Trans. Amer. Math. Soc.},
  volume={234},
  date={1977},
  number={2},
  pages={289--359},
  issn={0002-9947},
  review={\MR {0578730 (58 \#28261a)}, \MR {0578730 (58 \#28261b)}},
}

\bib{Fell62}{article}{
  author={Fell, J. M. G.},
  title={A Hausdorff topology for the closed subsets of a locally compact non-Hausdorff space},
  journal={Proc. Amer. Math. Soc.},
  volume={13},
  date={1962},
  pages={472--476},
  issn={0002-9939},
  review={\MR {0139135 (25 \#2573)}},
}

\bib{FD1}{book}{
  author={Fell, J. M. G.},
  author={Doran, R. S.},
  title={Representations of $^*$-algebras, locally compact groups, and Banach $^*$-algebraic bundles. Vol. 1},
  series={Pure and Applied Mathematics},
  volume={125},
  note={Basic representation theory of groups and algebras},
  publisher={Academic Press, Inc., Boston, MA},
  date={1988},
  pages={xviii+746},
  isbn={0-12-252721-6},
  review={\MR {936628 (90c:46001)}},
}

\bib{KR}{article}{
  author={Kruml, David},
  author={Resende, Pedro},
  title={On quantales that classify $C\sp \ast $-algebras},
  language={English, with French summary},
  journal={Cah. Topol. G\'eom. Diff\'er. Cat\'eg.},
  volume={45},
  date={2004},
  number={4},
  pages={287--296},
  issn={1245-530X},
  review={\MR {2108195 (2006b:46096)}},
}

\bib{Kumjian98}{article}{
  author={Kumjian, Alex},
  title={Fell bundles over groupoids},
  journal={Proc. Amer. Math. Soc.},
  volume={126},
  date={1998},
  number={4},
  pages={1115--1125},
  issn={0002-9939},
  review={\MR {1443836 (98i:46055)}},
  doi={10.1090/S0002-9939-98-04240-3},
}

\bib{Mulvey-enc}{article}{
  author={Mulvey, Christopher J.},
  title={Quantales},
  book={ editor={M. Hazewinkel}, title={The Encyclopaedia of Mathematics, third supplement}, publisher={Kluwer Acad. Publ.}, },
  date={2002},
  pages={312--314},
}

\bib{MP1}{article}{
  author={Mulvey, Christopher J.},
  author={Pelletier, Joan Wick},
  title={On the quantisation of points},
  journal={J. Pure Appl. Algebra},
  volume={159},
  date={2001},
  number={2-3},
  pages={231--295},
  issn={0022-4049},
  review={\MR {1828940 (2002g:46126)}},
}

\bib{MP2}{article}{
  author={Mulvey, Christopher J.},
  author={Pelletier, Joan Wick},
  title={On the quantisation of spaces},
  note={Special volume celebrating the 70th birthday of Professor Max Kelly},
  journal={J. Pure Appl. Algebra},
  volume={175},
  date={2002},
  number={1-3},
  pages={289--325},
  issn={0022-4049},
  review={\MR {1935983 (2003m:06014)}},
}

\bib{NT96}{article}{
  author={Nogura, Tsugunori},
  author={Shakhmatov, Dmitri},
  title={When does the Fell topology on a hyperspace of closed sets coincide with the meet of the upper Kuratowski and the lower Vietoris topologies?},
  booktitle={Proceedings of the International Conference on Convergence Theory (Dijon, 1994)},
  journal={Topology Appl.},
  volume={70},
  date={1996},
  number={2-3},
  pages={213--243},
  issn={0166-8641},
  review={\MR {1397079 (97f:54011)}},
  doi={10.1016/0166-8641(95)00098-4},
}

\bib{Paterson}{book}{
  author={Paterson, Alan L. T.},
  title={Groupoids, inverse semigroups, and their operator algebras},
  series={Progress in Mathematics},
  volume={170},
  publisher={Birkh\"auser Boston Inc.},
  place={Boston, MA},
  date={1999},
  pages={xvi+274},
  isbn={0-8176-4051-7},
  review={\MR {1724106 (2001a:22003)}},
}

\bib{RenaultLNMath}{book}{
  author={Renault, Jean},
  title={A groupoid approach to $C^{\ast } $-algebras},
  series={Lecture Notes in Mathematics},
  volume={793},
  publisher={Springer},
  place={Berlin},
  date={1980},
  pages={ii+160},
  isbn={3-540-09977-8},
  review={\MR {584266 (82h:46075)}},
}

\bib{Renault}{article}{
  author={Renault, Jean},
  title={Cartan subalgebras in $C^*$-algebras},
  journal={Irish Math. Soc. Bull.},
  number={61},
  date={2008},
  pages={29--63},
  issn={0791-5578},
  review={\MR {2460017 (2009k:46135)}},
}

\bib{Re07}{article}{
  author={Resende, Pedro},
  title={\'Etale groupoids and their quantales},
  journal={Adv. Math.},
  volume={208},
  date={2007},
  number={1},
  pages={147--209},
  issn={0001-8708},
  review={\MR {2304314 (2008c:22002)}},
}

\bib{Vietoris}{article}{
  author={Vietoris, Leopold},
  title={Bereiche zweiter Ordnung},
  journal={Monatsh. Math. Phys.},
  volume={32},
  date={1922},
  number={1},
  pages={258--280},
  review={\MR {1549179}},
}

\end{biblist}

\end{bibdiv}
\end{document}